    \RequirePackage{fix-cm}
    \documentclass[smallextended]{svjour3}       
    \smartqed  
    \usepackage[dvipdfmx]{graphicx}
    \makeatletter
    \@addtoreset{equation}{section}
    
    \makeatother
    
    \usepackage{amsmath}
    \usepackage{amssymb}
   \usepackage{algorithm}  
    \usepackage{algorithmic}
    \usepackage{color}
    \usepackage{ulem}
    \usepackage{url}
    \usepackage{cite}
    \usepackage{bm}
    \usepackage{mathrsfs}

    \newcommand{\Add}[1]{\textcolor{black}{#1}}
    
    \newcommand{\mumaxz}{\mu_{\max}^{\rm{z}}}
    \newcommand{\mumaxnz}{\mu_{\max}^{\rm{nz}}}
    \newcommand{\grad}{{\mathrm{grad }}}
    \newcommand{\gradf}{{\mathrm{grad } f}}
    \newcommand{\dist}{{\mathrm{dist}}}
    
    \newcommand{\injR}{{\mathrm{Inj}^R}}
    
    \newcommand{\Dist}{{\mathrm{Dist}}}
    \newcommand{\diffe}{{\mathrm{d}}}
    \newcommand{\bxk}{\overline{x_k}}
    \newcommand{\TM}{T \mathcal{M}}
    \newcommand{\mani}{\mathcal{M}} 

        \newcommand{\tred}{\textcolor{black}}
    \newcommand{\tblue}{\textcolor{black}}
    
\begin{document}
    
    \title{Riemannian Levenberg-Marquardt Method with Global and Local Convergence Properties
    }
    
    \titlerunning{Riemannian LM Method with Global and Local Convergence Properties}        
    
    \author{Sho Adachi        \and
      Takayuki Okuno \and
      Akiko Takeda
    }
    
    
    \institute{S. Adachi  \at Graduate School of Information Science and Technology, University of Tokyo, Tokyo, Japan \\ \email{adachi-syo410@g.ecc.u-tokyo.ac.jp}           
               \and
               T. Okuno \at Faculty of Science and Technology,
Seikei University, Tokyo, Japan;  Center for Advanced Intelligence Project, RIKEN, Tokyo, Japan 
      \\ \email{takayuki-okuno@st.seikei.ac.jp}
      \and 
      A. Takeda \at 
Graduate School of Information Science and Technology, University of
Tokyo, Tokyo, Japan;
Center for Advanced Intelligence Project, RIKEN, Tokyo, Japan
     \\ \email{takeda@mist.i.u-tokyo.ac.jp} }   
    \date{Received: date / Accepted: date}

    \maketitle

    \begin{abstract}
      We extend the Levenberg-Marquardt method on Euclidean spaces to Riemannian manifolds. 
      Although a Riemannian Levenberg–Marquardt (RLM) method was proposed by Peeters in 1993, to the best of our knowledge, there has been no analysis of theoretical guarantees for global and local convergence properties.
      As with the Euclidean LM method, how to update a specific parameter known as the ``damping parameter'' has significant effects on its performances. We propose a trust-region-like approach for determining the parameter. 
       We evaluate the worst-case iteration complexity to reach an $ \epsilon$-stationary point, and also prove that it has desirable local convergence properties under the local error-bound condition. Finally, we demonstrate the efficiency of our proposed algorithm by numerical experiments.
    \end{abstract}

\section{Introduction} \label{intro}
    
Optimization problems over Riemannian manifolds, manifolds equipped with smoothly varying positive definite symmetric metrics at every point, have been studied intensively. On Riemannian manifolds, we can construct counterparts of a variety of basic concepts in Euclidean optimization such as gradient and Hessian. With these extended concepts, classical unconstrained optimization methods on Euclidean spaces such as the steepest descent method and the Newton method have been generalized to Riemannian manifolds \cite{RSDRN}.
For example, there exist Riemannian quasi-Newton methods 
\cite{Riemanniantrustregion2,Riemannianquasi-Newton1,Riemannianquasi-Newton2,RiemannianCG1}, Riemannian conjugate gradient methods
\cite{RSDRN,RiemannianCG1,RiemannianCG2,RiemannianCG3}, Riemannian trust region methods \cite{Riemanniantrustregion1,RSDRN,RTRconvergencerate,Riemanniantrustregion2,
Riemanniantrustregion3,baker2008implicit,boumal2015riemannian,li2021nonmonotone
}, and so forth.
Moreover, in the last few years, studies on constrained optimization methods on Riemannian manifolds have been advanced remarkably, too. For instance, 
Riemannian augmented Lagrangian methods \cite{cliu,yamakawa},
Riemannian SQP methods \cite{schiela2021sqp,obara}, and Riemannian interior point methods \cite{lai2022superlinear} have been proposed together with rigorous convergence analysis.

In this paper, we consider the nonlinear least square problems over an $n$-dimensional connected Riemannian manifold $ (\mani,\langle \cdot, \cdot \rangle)$ , i.e.,  
\begin{align}
  \begin{split}
  \label{P}
  \underset{x\in\mani}{\text{minimize}} \quad & f(x) := \frac{1}{2}\| F(x) \|^2 (= \frac{1}{2}\sum_{i=1}^{m} F_i(x)^2), 
  \end{split}
\end{align}
where $F_i \colon \mani \rightarrow \mathbb{R} $ ($i=1,\dots,m$) are continuously differentiable functions and $ F: \mani \rightarrow \mathbb{R}^m $ is defined as  $ F = \left( F_1,F_2,\dots,F_m\right)^T $. 
The problem \eqref{P} in Euclidean spaces, namely, when $\mani=\mathbb{R}^n$, has many classic, but important applications
ranging from inverse problems \cite{inverseproblem}, regressions \cite{regression}, to systems of nonlinear equations \cite{nonlineareq}.
In addition, some recent applications such as the CP decomposition of tensors \cite{TAP}, the low-rank matrix completion \cite{low-rankmatrixcompletion}, the Fr\'{e}chet mean \cite{originoffrechetmean}, or the geodesic regression \cite{geodesicregressionproposed1,geodesicregressionproposed2} are formulated as \eqref{P} with some more general Riemannian manifold $\mani $.

In Euclidean spaces, 
optimization methods for nonlinear least-square problems have been studied extensively. Especially, the Gauss-Newton (GN) method and the Levenberg-Marquardt (LM) method 
are the most popular methods specialized in solving this type of problem.
Their common strength is that the fast local convergence can be attained without computing the Hessian of $F$ which is often costly when $n$ is large.  Indeed, under some assumptions, both the methods admit local quadratic convergence for the zero-residual case which implies $\min_{x\in\mathbb{R}^n}\|F(x)\|^2=0$. 
The GN method \cite{Bertsekas} generates a search direction 
in each iteration by solving the linear equation equivalent to the natural approximation problem 
$\min_{d\in \mathbb{R}^n}\,\|F(x)+J(x)d\|^2$, where 
$J(x)$ denotes the Jacobian of $F$.  
However, it is often pointed out that the GN method fails to work upon confronting the ill-conditioned linear equations derived from the rank deficiency of $J(x)$. 
On the other hand, the LM method was developed to cope with this matter \cite{Levenberg,Marquardt} by solving the regularized linear equations equivalent to $\min_{d\in \mathbb{R}^n}\,\|F(x)+J(x)d\|^2+\frac{\rho}{2}\|d\|^2$, where $\rho$ is a positive parameter called \textit{damping parameter}. 
Various versions of LM methods have been proposed so far \cite{Behling,marumo2023majorization,Bergou,globalRLMtrust,Yamashita} and their theoretical and practical performances vary mainly depending on a manner of updating the damping parameter. 

In summary, the Euclidean LM method enjoys several nice properties.
It admits a local quadratic convergence for the zero-residual case under a local {\it error-bound} condition, which is weaker than the nonsingularity condition of the Jacobian $J$ at a solution \cite{Yamashita}. For the nonzero-residual case, the local linear convergence of the LM method was also shown under the local error-bound condition by \cite{locallinear}.  
Moreover, the global convergence complexity was 
studied, e.g. \cite{globalRLM,globalRLMtrust,Bergou}. 
    
Some researchers worked with GN and LM methods 
on Riemannian manifolds for solving \eqref{P}. For example, the basic Riemannian GN (RGN) method is described in the textbook, \cite[Section 8.4.1]{RSDRN}. In \cite{TAP}, it was customized for solving a certain tensor-decomposition problem.
Although the local convergence properties were established in \cite{RGN} under some assumptions on the Jacobian of $F$, it is not equipped with any global convergence property. 
The first Riemannian LM (RLM) method was considered in \cite{RLM0}, but any theoretical results were not presented therein.
Moreover, although \cite[Section 8.4.1]{RSDRN} pointed out that a combination of the trust region method and the RGN method can be regarded as the  RLM method, no specific algorithm is presented there. In this paper, we propose the first RLM method equipped with both global and local convergence guarantees by developing a specific trust-region-like manner of tuning the damping parameter.

\subsection{Our contribution}
Our contribution is summarized as follows:
\begin{enumerate}
\item {\bf Development of the RLM method}: 
Even though a Riemannian version of the LM method was considered in \cite{RLM0}, it only states how to find an LM-like search direction independently from local coordinates. We characterize the search direction as the tangent vector minimizing  \textit{subproblem} of \eqref{P} and thus, it is intrinsically independent of the choice of local coordinates.
Our RLM method is different from the one in \cite{RLM0} especially in the update manner for the damping parameter. 
\item {\bf Theoretical guarantees for the RLM method}:
  Our method has theoretical convergence guarantees: global iteration complexity and local convergence rates.
\begin{itemize}
\item 
  Our method is globally convergent and 
  the worst-case iteration complexity to reach an $\epsilon $-stationary point is $O\left( \log{(\epsilon^{-1})}\epsilon^{-3}\right)$
under standard assumptions such as $L$-smoothness, explained later in Section~\ref{global}. 
\item 
    The local convergence analysis evaluates the algorithm's behavior around a stationary point $x^*$, and
  the convergence rate differs depending on  whether the residual $f(x^*)$ is zero or not. The former case is often called the {\it zero-residual case}, while the latter is the {\it nonzero-residual case}.
We extend the local error bound condition, which is a standard assumption for Euclidean LM methods, to Riemannian manifolds. Under this condition, we prove that the proposed RLM has the quadratic local convergence for zero-residual cases and the linear one for nonzero-residual cases.
\end{itemize}
\end{enumerate}
 Unlike Euclidean setting,
the local convergence analysis is complicated because
  the search direction $s_k$ does not generally satisfy $ \| s_k \|_{x_k} = \dist (x_k,x_{k+1})$ except in special circumstances\footnote{More specifically, the search direction $s_k$ does not satisfy $ \| s_k \|_{x_k} = \dist (x_k,x_{k+1})$ unless $s_k = \log_{x_k}(x_{k+1}) $ holds.
  Here, $\dist : \mani \times \mani \rightarrow \mathbb{R}_{\geq 0} $ denotes the Riemannian distance 
  and $\log_{x_k} : \mani \rightarrow T_{x_k} \mani $ denotes the \textit{logarithmic map} at $x_k$}. This prevents us from applying the standard approach for the local convergence analysis of LM methods in Euclidean spaces to our RLM method. However, we settle this issue by introducing an inequality on the Riemannian distance and the norm of tangent vectors obtained by the inverse retraction.
%
%
Finally, let us
{make comparison with 
two Riemannian methods which are related to the RLM method,
the adaptively quadratically regularized Newton (ARN) method
\cite{quadraticallyreguralizedNewton} and 
the Riemannian trust region (RTR) method \cite{Riemanniantrustregion1,RSDRN,RTRconvergencerate,Riemanniantrustregion2,
Riemanniantrustregion3,baker2008implicit,boumal2015riemannian,li2021nonmonotone
}:
\begin{itemize}
\item The ARN method solves a sequence of  quadratic subproblems with proximal regularization and selects the regularization parameter adaptively.
This ARN is quite similar to the RLM in that the regularization technique is employed.
However, in the article\,\cite{quadraticallyreguralizedNewton}, 
the global complexity is not derived, and the assumptions for the local quadratic convergence of the ARN, which are set in \cite{quadraticallyreguralizedNewton}, include the nonsingularity of the Jacobian; the regularity assumption is stronger than ours, i.e., the local error bound. It may be worth mentioning that the ARN 
\cite{quadraticallyreguralizedNewton}
is limited to manifolds embedded to Eucledian spaces, and the regularization is induced from the squared Euclidean 2-norm. In contrast, our RLM is not the case, and the regularization is described with the norm induced from the Riemannian metric. As a result, search directions of the RLM are determined independently from local coordinates. 
\item The RTR solves a sequence of quadratic subproblems subject to  the so-called trust region on the tangent space of the Riemannian manifold. The trust region radius is tuned so that the quadratic subproblem is a good approximation to the original problem. 
As observed from the Karush-Kuhn-Tucker conditions for the quadratic subproblem, the trust-region scheme has an effect similar to the regularization technique.  However, as well as the ARN method, 
the regularity assumption 
that hessians are nondegenerate at a solution is set in \cite{Riemanniantrustregion1,
baker2008implicit,Riemanniantrustregion3,li2021nonmonotone}
for the local convergence. 
In particular, an attractive point of the RLM is that quadratic convergence is achieved without using Hessians. 
As regards complexity of the global convergence, 
the analysis \cite{RTRconvergencerate} of the RTR method is similar to ours in the present paper. Nevertheless, it is difficult to translate the lemmas and propositions established in \cite{RTRconvergencerate} as those for our RLM.
\end{itemize}
In the numerical experiments, we will make numerical comparison with the RTR and ARN methods, and show that the RLM performs very well.   
}

\subsection{Organization of the paper} 
  The rest of this paper is organized as follows:
In Section \ref{chap:algorithm}, we propose a new RLM method, and
in Section~\ref{global}, we show that the RLM method has a global convergence property and further evaluate the worst-case iteration complexity to reach an $\epsilon$-stationary point.
In Section~\ref{chap:local}, we analyze the local behavior of the proposed RLM method.
In Section~\ref{chap:experiment}, 
we conduct some numerical experiments of the RLM method, and
in Section~\ref{chap:conclusion}, we conclude this paper with some remarks.

{
  \subsection{Notations and terminologies}
For a given Riemannian manifold $(\mathcal{M},\langle \cdot, \cdot \rangle) $, $ T_x \mathcal{M}$ denotes the tangent space to $ \mathcal{M}$ at $x$ and $ \TM $ is the tangent bundle of $\mathcal{M}$, i.e. $ \TM := \coprod_{x \in \mathcal{M}} T_x \mathcal{M}$. Let $ 0_x$ be the zero vector of $ T_x \mathcal{M}$ as a vector space.
We denote by $ \langle \cdot , \cdot \rangle_x $ the inner product induced by the Riemannian metric at the point $x \in \mathcal{M}$ and
$ \| \cdot \|_x$ is the norm induced by the inner product, i.e., $ \|v \|_x := \sqrt{ \langle v, v \rangle_x}$ for $ v \in T_x \mathcal{M}$. 
For the sake of brevity, the subscript $x$ is dropped from $ \| \cdot \|_x$ when it is clear from the context.

Given manifolds $ \mani$ and $\mathcal{N}$ together with 
a smooth map $g: \mani \rightarrow \mathcal{N}$, 
$\diffe g$ denotes the differential of $g$, namely, for all $p \in \mani$, $ \diffe g(p)$ is the linear map from $ T_p \mani$ to $ T_{g(p)} \mathcal{N}$ such that 
$\diffe g(p) [v] = \left.\frac{\diffe (g \circ c)  }{\diffe t }\right|_{t=0} $ holds for all $ v \in T_{p} \mani $ 
where $ c : (-\epsilon,\epsilon) \rightarrow \mani$ ( $\epsilon > 0 $ ) is any smooth curve satisfying $c(0) = p $ and $\left. \frac{\diffe c(t) }{\diffe t}\right|_{t=0} = v$.   

We next introduce a {\it retraction}, playing an important role for Riemannian optimization algorithms to determine an iteration point on the manifold along a given tangential direction.
\begin{definition}\label{retraction defi}
  A retraction $R$ is a smooth map from $ \TM$ to $\mathcal{M} $ with the following properties. Let $R_x$ denote the restriction of $R$ to $x \in \mathcal{M}$.
  \begin{enumerate}
\item $R_x(0_x) = x$ 
for all $x \in \mathcal{M}$.\\
\item For all $x \in \mathcal{M} $, the differential map $ \diffe R_x(0_x) : T_x \mathcal{M} \rightarrow T_x \mathcal{M}$ is the identity map on $T_x \mathcal{M}$.
 \end{enumerate}
\end{definition}

Next, we define some notations and terminologies concerning the function $F$
in \eqref{P}.
We refer to the following linear map as the Jacobian matrix of $F$ at $x \in \mani$:
\begin{eqnarray*}
\begin{array}{rccc}
    J(x) \; \colon & T_{x} \mathcal{M}                  &\longrightarrow& \mathbb{R}^m                  \\
            & \rotatebox{90}{$\in$}&               & \rotatebox{90}{$\in$} \\
            & s                & \longmapsto   & \left( \begin{array}{c}
                \langle \text{grad}F_1(x),s \rangle_{x} \\
                \vdots \\
                \langle \text{grad}F_m(x),s \rangle_{x}
            \end{array}
            \right),
\end{array}
\end{eqnarray*}
where $\grad F_i (x)$ is the Riemannian gradient of $F_i$ at $x$ for $1 \leq i \leq m$. 
We denote the adjoint operator of $J(x)$ by $ J(x)^* : \mathbb{R}^m \rightarrow T_{x} \mani$, i.e.,
\begin{eqnarray}
        \langle J(x)^* u, v \rangle_{x} &=& \langle u, J(x) v \rangle \label{def of adj}
\end{eqnarray}
for all  $(u,v) \in \mathbb{R}^m\times T_{x} \mathcal{M}$, where the inner product in the right-hand side is the canonical inner product of $ \mathbb{R}^m$.
Finally,  we define the norm 
$ \| J(x) \|$ of the Jacobian matrix $J(x)$ for $x\in \mani$
by means of the operator norm, namely, 
\begin{eqnarray*}
    \| J(x) \| &:=& \max_{v \in T_x \mani \setminus \{0_x\}} \frac{\| J(x)v \|}{\|v\|_x},
\end{eqnarray*}
where the norm in the numerator of the right-hand side represents the Euclidean norm. Clearly, $ \|J(x) \|$ is equal to the square root of the maximum eigenvalue of $ J(x)^* J(x):T_x\mani\to T_x\mani$.
}


   \section{Proposed Riemannian Levenberg-Marquardt method} \label{chap:algorithm}
    
    In this section, we describe how the search direction of the RLM is determined in each iteration and present the specific pseudo-code for the RLM. In addition, we will characterize the search direction which plays important roles in theoretical analysis. Our formulation of the RLM can be considered as a natural generalization of the LM in Euclidean spaces. Let $\{x_k\}\subseteq \mathcal{M}$ denote a sequence generated by the proposed RLM, and let $J_k :=J(x_k)$.

    The pseudo-code of our proposed method is  Algorithm~\ref{alg1}. This trust-region-like updating scheme of the damping parameter is proposed by \cite{Bergou} in the Euclidean setting. Below, we show the details of the algorithm.

    \begin{algorithm}[H]
        \caption{RLM method }
        \label{alg1}
        \begin{algorithmic}[1]
        \renewcommand{\algorithmicrequire}{\textbf{Input:}}
        \renewcommand{\algorithmicensure}{\textbf{Output:}}
        \REQUIRE $x_0 \in \mathcal{M} , \eta \in (0,1), \mu_{\min}>0, \beta >1, \text{flag}^{\text{nz}} \in \{ \text{true},\text{false}\}$ \; \; 
        \ENSURE stationary point of \eqref{P}
        \STATE $\mu_0 \leftarrow \mu_{\min}$ , $\bar{\mu} \leftarrow \mu_0 $
       \STATE $k \leftarrow 0$
        \WHILE{not convergence}
        \STATE compute $F(x_k), J_k$
        \STATE compute $s_k$ by solving \eqref{linear eq of LM} with $\lambda_k = \mu_k \| F(x_k)\|^2$.
        \STATE compute $\rho_k := \frac{f(x_k)-f(R_{x_k}(s_k))}{\frac{1}{2}(\theta^k(0_{x_k})-\theta^k(s_k))} $
        \IF{$ \rho_k \geq \eta$}
        \STATE $ x_{k+1} \leftarrow R_{x_k}(s_k)$, $\bar{\mu} \leftarrow \mu_k$
        \IF{$\text{flag}^{\text{nz}}$}
        \STATE $ \mu_{k+1} \leftarrow \bar{\mu}$
        \ELSE
        \STATE $ \mu_{k+1} \leftarrow \max{\left(\mu_{\min},\label{line8} \frac{\bar{\mu}}{\beta} \right) }$
        \ENDIF
        \ELSE 
        \STATE $ x_{k+1} \leftarrow x_k \; , \; \mu_{k+1} \leftarrow \beta \mu_k$
        \ENDIF
        \STATE $ k \leftarrow k+1$
        \ENDWHILE
        \RETURN $x_k$
        \end{algorithmic}
    \end{algorithm}

    \subsection{Subproblem for problem~\eqref{P}}
    Given $ \lambda_k >0$, define $\theta^k : T_{x_k} \mani \rightarrow \mathbb{R} $ as 
    \begin{eqnarray}
        \theta^k(s) &:=& \| F(x_k)+J_k s \|^2 + \lambda_k \|s\|^2_{x_k} \label{original theta_k_s} \\
        &=&  \| F(x_k) \|^2 + 2 \langle \grad f(x_k), s \rangle_{x_k} + \| J_k s \|^2 + \lambda_k \| s \|_{x_k}^2, \label{theta_k_s}
    \end{eqnarray}
   where $ \lambda_k$ is called a damping parameter and we will explain how to update it later. The term 
   $F(x_k) + J_k s $ in \eqref{original theta_k_s} 
   corresponds to the linearization  of $F(R_{x_k}(s))$ at $x_k$, where $R$ is the retraction defined in
   Definition~\ref{retraction defi}.
   It is worth noting that 
   $\| F(x_k)+J_k s \|^2$ is equipped with the Euclidean norm, while 
   $\|s\|^2_{x_k}$ is measured by the norm induced by the Riemannian metric.
   
 We solve the following problem as the subproblem of ~\eqref{P} at $x_k$ 
    \begin{align}
        \begin{split}
        \label{sub problem}
    \underset{s\in T_{x_k} \mani}{\text{minimize}} \quad & \theta^k(s),
        \end{split}
    \end{align}
    and denote the optimal solution by $s_k$, namely
        \begin{eqnarray}
        s_k := \underset{s \in T_{x_k} \mani}{\text{ arg min }} \theta^k(s). \label{ s_k = argmin theta}
    \end{eqnarray}
        This problem is strongly convex on $T_{x_k} \mani$, and thus has a unique optimum. 
        We employ the solution $s_k$ as the search direction at $x_k$. 
Through the stationary condition of \eqref{sub problem}, 
$s_k$ is characterized as a solution of a certain linear equation 
as in the following proposition. This relationship is in fact non-trivial because of $\|\cdot\|_{x_k}$ in the function $\theta^k$.
{\begin{proposition} \label{charcterization of search direction}      
The tangent vector $s_k$ solves problem\,\eqref{sub problem} if and only if it satisfies
    \begin{eqnarray}\label{linear eq of LM}
        (J_k^* J_k + \lambda_k I_k ) s_k &=& - {J}^*_k F(x_k) \notag \\
        &=& - \grad f(x_k),
    \end{eqnarray}
    where $ I_k$ denotes the identity map on $ T_{x_k} \mani$. 
    In particular,   
    the equation \eqref{linear eq of LM} has a unique solution.  
    \end{proposition}}
\begin{proof}
  Since the latter assertion is trivial as the linear operator $J_k^* J_k + \lambda_k I_k $ is positive definite, namely, 
 $\langle v, (J_k^* J_k + \lambda_k I_k ) v  \rangle_{x_k} > 0$ for all $ v \in T_{x_k} \mani\setminus\{0_{x_k}\}$, 
  we prove only the former one.

  Let $(U; x^1,\dots,x^n) $ be an arbitrary coordinate neighborhood containing $x_k $. Let $G$ be the matrix representation of the Riemannian metric at $x_k$ under this local coordinate, i.e., $  G = \left(g_{ij}\right)_{1 \leq i,j \leq n}$ where $g_{ij} := \langle \frac{\partial}{\partial x^i},\frac{\partial}{\partial x^j} \rangle_{x_k}$. 
    Let $s$ be an arbitrary element of $ T_{x_k} \mani$ and $ \tilde{s}$ be its local coordinate expression. In a similar way, we define $ \tilde{v}$ as the local coordinate expression of $J_k^* F(x_k) $ and $A$ as the matrix representation of $ J_k^*J_k + \lambda_k I_k$ under this local coordinate. Then, we have
    \begin{eqnarray} \label{ local exp of theta}
        \theta^k(s) &=& \|F(x_k) \|^2 + 2 \langle F(x_k), J_ks \rangle + \langle J_ks, J_ks \rangle + \lambda_k \langle s,s \rangle_{x_k} \notag \\
        &=& \|F(x_k) \|^2 + 2 \langle J_k^*F(x_k), s \rangle_{x_k} + \langle J_k^*J_k s , s \rangle_{x_k} + \lambda_k \langle s,s \rangle_{x_k} \notag \\
        &=& \|F(x_k) \|^2 + 2 \tilde{v}^T G \tilde{s} + \langle (J_k^* J_k +\lambda_k I_k)s ,  s \rangle_{x_k} \notag \\
        &=&  \|F(x_k) \|^2 + 2 \tilde{v}^T G \tilde{s} + \tilde{s}^T G A \tilde{s} ~~ (=: \tilde{\theta}^k (\tilde{s})).
    \end{eqnarray}
    We show that $GA$ is symmetric and positive-definite.
    Denoting the local coordinate expression of $J_k$ as $\tilde{J}_k$, we can express that of $J_k^*$ as $ G^{-1} {\tilde{J}_k}^T$
    and hence $A =  G^{-1} {\tilde{J}_k}^T \tilde{J}_k + \lambda_k I$ holds where $I$ denotes the $n$-dimensional identity matrix. 
    Consequently, 
    \begin{eqnarray}
        GA  &=& G \left(G^{-1} {\tilde{J}_k}^T \tilde{J}_k + \lambda_k I\right) = {\tilde{J}_k}^T \tilde{J}_k + \lambda_k G, \label{GA} \\
        A^T G &=& \left( {\tilde{J}_k}^T \tilde{J}_k G^{-1} + \lambda_k I \right) G = {\tilde{J}_k}^T \tilde{J}_k + \lambda_k G, \notag
    \end{eqnarray}
    and thus, $ GA = (GA)^T $is shown and the positive-definiteness of $GA$ immediately follows from \eqref{GA}.
    Consequently, the optimality condition of minimizing $ \tilde{\theta}^k (\tilde{s})$ is equivalent to the stationary condition
$\frac{\partial}{\partial \tilde{s}} \tilde{\theta}^k(\tilde{s}) = 0$, which is written as $G \left( A \tilde{s} + \tilde{v} \right) = 0$.
      Furthermore, by the positive definiteness of $G$, this is equivalent to
    \begin{equation}\label{local exp of stationary condition}
      A \tilde{s}  = - \tilde{v}.
    \end{equation}
    Considering the coordinate-independent form of (\ref{local exp of stationary condition}), we obtain $ (J_k^* J_k + \lambda_k I_k) s = - J_k^* F(x_k)$.
Therefore, we have reached the desired conclusion.
     \hfill$\Box$
\end{proof}    
 \subsection{How to update the damping parameter}
The damping parameter $ \lambda_k$ in \eqref{original theta_k_s} controls the step length
and needs to be chosen in a manner reflecting how trustworthy the subproblem is: when the subproblem \eqref{sub problem} is close to the original problem \eqref{P}, $ \lambda_k$ is set relatively small and otherwise, it is set relatively large.
In a similar manner to the trust region method\,\cite{nonlineareq}, we evaluate the quality of the solution $s_k$ of \eqref{sub problem} in terms of $\rho_k$ defined by
    \begin{eqnarray} \label{ def of rho_k}
        \rho_k := \frac{f(x_k)-f(R_{x_k}(s_k))}{\frac{1}{2}\left(\theta^k(0_{x_k})-\theta^k(s_k)\right)}.
    \end{eqnarray}
    \tred{Let $\eta\in (0,1)$ be a prefixed constant.}
\tred{If $ \rho_k\ge \eta$ holds, we judge the subproblem is trustworthy, and then update $ x_{k+1}$ as $ x_{k+1} = R_{x_k}(s_k)$.}
In this case, we refer to the $k$-th iteration as a \textit{successful} iteration. Otherwise, it is called \textit{unsuccessful}. We \tred{reject $ R_{x_k}(s_k)$ and set $x_{k+1} = x_k$. After setting $ \lambda_k$ larger, we solve the subproblem again and check whether $ \rho_k\ge \eta$ or not.}

\tred{
  We explain how to tune $\lambda_k$. First, we set
  \begin{eqnarray} \label{update_lambda}
    \lambda_k = \mu_k \| F(x_k) \|^2,
 \end{eqnarray}
    which is a standard choice so as to achieve locally fast convergence in the recent Euclidean LM methods\,\cite{Bergou,mukFk21,mukFk22,mukFk23,mukFk24}. }
The positive parameter $\mu_k $ is updated as shown in  Algorithm~\ref{alg1}.
To ensure global convergence in Section \ref{global} and local convergence for the zero-residual case in Section \ref{local zero}, we set $ \mu_{k+1} = \max{\left(\mu_{\min}, \frac{\bar{\mu}}{\beta} \right) } $, while we set $ \mu_{k+1} = \bar{\mu}$ to establish local convergence for the nonzero-residual case in Section \ref{local nonzero}. The parameter $\text{flag}^{\text{nz}} $ in Algorithm~\ref{alg1} specifies which updating manner for $\mu_k $ is applied and is supposed to be set false except for the nonzero-residual case
in Section \ref{local nonzero}.

\begin{remark}
    Even when we update $\{\mu_k\} $ by $ \mu_{k+1} = \bar{\mu}$, we can still ensure the global convergence property. Meanwhile, the iteration complexity analysis in Section~\ref{iteration complexity analysis}, however, can depend on $ \text{flag}^{\text{nz}}$ in Algorithm~\ref{alg1}. In this paper, we only discuss the iteration complexity of Algorithm~\ref{alg1} with $\text{flag}^{\text{nz}} = \text{false} $.
\end{remark}

    \section{Analysis on global convergence and iteration complexity} \label{global}    
    In this section, 
    \tblue{we set $\text{flag}^{\text{nz}} ={\rm false}$ in Algorithm~\ref{alg1}}. We prove that Algorithm~\ref{alg1} has a global convergence property and then, analyze its iteration complexity. We begin with \tred{giving} some assumptions and lemmas.
        
    We define $\mathcal{S}$ as the set of successful iterations, namely, 
      \begin{eqnarray*}
        \mathcal{S} := \{k \in \{0,1,2,\dots\} \; | \; \rho_k \geq \eta \},
      \end{eqnarray*}
    where $\eta$ is the constant in Algorithm~\ref{alg1}. 
    For the sake of convenience in proofs, we express $\mathcal{S} $ as 
   \begin{equation*}
   \mathcal{S} = \{k(0),k(1),\dots\}.
   \end{equation*}
    \begin{lemma}\label{Cauchy step}
        The $k$-th search direction $s_k$ satisfies $\theta^k(0)-\theta^k(s_k) \geq \frac{\| \grad f(x_k) \|^2}{\|{J_k}\|^2 + \lambda_k}$. 
    \end{lemma}
    
    \begin{proof}
        Define
        \begin{eqnarray} \label{ defi of Cauchy step}
            s_k^c &:=& - \frac{\| \grad f(x_k) \|^2}{\langle \grad f(x_k),({J_k}^* J_k + \lambda_k I_k ) \grad f(x_k) \rangle} \grad f(x_k).
        \end{eqnarray}
        First, we have
        \begin{alignat}{2}
            & \; \; \; \; \; \theta^k(0)-\theta^k(s_k^c) \notag \\
            & = -2\langle \text{grad}f(x_k),s_k^c \rangle - \| J_k s_k^c \|^2 - \lambda_k \|s_k^c \|^2
            & \quad & \text{ (by \eqref{theta_k_s}) } \notag \\
            &= \frac{2\| \text{grad}f(x_k) \|^4}{\langle \text{grad}f(x_k),({J_k}^* J_k + \lambda_k I_k ) \text{grad}f(x_k) \rangle} - \left( \| J_k s_k^c \|^2 + \lambda_k \|s_k^c \|^2 \right), \label{Appendix D 1}
        \end{alignat}
        where \eqref{ defi of Cauchy step} is used in the last equality.
        \tred{Moreover, we obtain}
    \begin{eqnarray*}
     && \|J_k s_k^c \|^2 + \lambda_k \|s_k^c \|^2 \\
     &=&  \frac{\| \text{grad}f(x_k) \|^4}{\langle \text{grad}f(x_k),({J_k}^* J_k + \lambda_k I_k ) \text{grad}f(x_k) \rangle ^2 } \left( \langle  J_k\text{grad}f(x_k), J_k \text{grad}f(x_k) \rangle + \lambda_k \langle \grad f(x_k), \grad f(x_k) \rangle  \right)  \\
     &=& \frac{\| \text{grad}f(x_k) \|^4}{\langle \text{grad}f(x_k),({J_k}^* J_k + \lambda_k I_k ) \text{grad}f(x_k) \rangle ^2 } \langle \text{grad}f(x_k), ({J_k}^*J_k + \lambda_k I_k)\text{grad}f(x_k) \rangle \\
     &=& \frac{\| \text{grad}f(x_k) \|^4}{\langle \text{grad}f(x_k),({J_k}^* J_k + \lambda_k I_k ) \text{grad}f(x_k) \rangle },
    \end{eqnarray*}
    where \eqref{ defi of Cauchy step} is applied in the first equality.
    Then, \tred{from \eqref{Appendix D 1}} we obtain
    \begin{eqnarray*}
     \theta^k(0)-\theta^k(s_k^c) &=& \frac{\| \text{grad}f(x_k) \|^4}{\langle \text{grad}f(x_k),({J_k}^* J_k + \lambda_k I_k ) \text{grad}f(x_k) \rangle } \\
     & \geq & \frac{\| \text{grad}f(x_k) \|^2}{\|{J_k}\|^2 + \lambda_k}.
    \end{eqnarray*}
    
    From the above inequality and the fact that $s_k = \underset{s \in T_{x_k} \mani}{\text{arg min }} \theta^k(s)$, it follows that
    \begin{eqnarray} \label{inequ from Cauchy step}
        \theta^k(0)-\theta^k(s_k) \geq \frac{\| \text{grad}f(x_k) \|^2}{\|{J_k}\|^2 + \lambda_k}.
    \end{eqnarray}
    \hfill$\Box$ 
    \end{proof}
    
    \Add{The following lemma will be used not only in this section but also in Section~\ref{chap:local}}
    \begin{lemma} \label{lemm for two inequs}
        The solution $s_k$ of \eqref{linear eq of LM} satisfies the following:
        \begin{eqnarray}
            && \|s_k \| \leq \frac{\|\grad f(x_k) \|}{\lambda_k}, \label{s_k grad rela} \\
            && -\langle \grad f(x_k), s_k \rangle \geq \frac{\| \grad f(x_k) \|^2 }{\|J_k\|^2 + \lambda_k}. \label{inner prod of grad s_k}
        \end{eqnarray}
    \end{lemma}
    
    \begin{proof}
        First, we prove \eqref{s_k grad rela}. By \eqref{linear eq of LM}, $\| \grad f(x_k) \|^2 $ satisfies
        \begin{eqnarray*}
            \| \grad f(x_k) \|^2 
            &=& \langle (J_k^* J_k + \lambda_k I_k)s_k, (J_k^* J_k + \lambda_k I_k)s_k \rangle \\
            &=& \| (J_k^* J_k)s_k \|^2 + 2 \lambda_k \| J_ks_k \|^2 + \lambda_k^2 \| s_k \|^2 \\
            & \geq & \lambda_k^2 \| s_k \|^2,
        \end{eqnarray*}
        which leads to \eqref{s_k grad rela}. 
    
        Next, we show \eqref{inner prod of grad s_k}. \tred{Since $s_k$ can be written as $ s_k = - \left( J_k^* J_k + \lambda_k I_k\right)^{-1} \grad f(x_k)$ from \eqref{linear eq of LM},}
        we obtain 
        \begin{eqnarray*}
            - \langle \grad f(x_k), s_k \rangle &=& \langle \grad f(x_k), \left( J_k^* J_k + \lambda_k I_k\right)^{-1}\grad f(x_k) \rangle \\
            & \geq & \frac{\| \grad f(x_k) \|^2 }{\|J_k\|^2 + \lambda_k},
        \end{eqnarray*}
        which proves \eqref{inner prod of grad s_k}.
        \hfill$\Box$ 
    \end{proof}
    
    \begin{assume}\label{assume global 1} 
        The Jacobian matrix $J : T\mani \rightarrow \mathbb{R}^m $ of $F$ and its adjoint $J^*$ are bounded on 
        $\mathcal{L}(x_0) := \{ x \in \mathcal{M} \; | \; f(x) \leq f(x_0)\}$,
        i.e., there exists $  M>0 $ such that $ \max{ \{ \| J(x) \|, \| J(x)^* \| \} } \leq M $ holds for all $ x \in \mathcal{L}(x_0)$. \\
    \end{assume}

    \subsection{Global convergence}
    \tred{Now, the global convergence theorem of the RLM is presented below.} Before moving to the main theorem, we show
    the following lemma.
    
    \begin{lemma} \label{mu_k s_k is bounded}
        \Add{Under Assumption~\ref{assume global 1}}, if $\underset{j \rightarrow \infty}{\liminf} \| F(x_{k(j)}) \| > 0 $, then $\underset{j \rightarrow \infty}{\limsup}\; \mu_{k(j)} \|s_{k(j)}\| < \infty$ holds.
    \end{lemma}
    
    \begin{proof}
        Using \eqref{s_k grad rela} \tred{and $\mu_k=\frac{\lambda_k}{\|F(x_k)\|^2}$}, we have $ \mu_{k(j)} \|s_{k(j)}\| \leq \frac{ \| \grad f(x_{k(j)}) \|}{\|F(x_{k(j)}) \|^2}$. Moreover, we have $\| \grad f(x_{k(j)}) \| = \| {J^*_{k(j)}} F(x_{k(j)}) \| \leq \|{J^*_{k(j)}} \|  \| F(x_{k(j)})\| \leq M \| F(x_{k(j)}) \| $ from Assumption \ref{assume global 1}. 
        Combining \tred{these relationships}, we obtain $\mu_{k(j)} \|s_{k(j)}\| \leq \frac{M}{\| F(x_{k(j)})\|} $. Therefore, $ \underset{j \rightarrow \infty}{\limsup}\; \mu_{k(j)} \|s_{k(j)}\| \leq M \underset{j \rightarrow \infty}{\limsup} \frac{1}{\| F(x_{k(j)}) \|} = \frac{M}{ \underset{j \rightarrow \infty}{\liminf} \| F(x_{k(j)}) \|} < \infty$.
        \hfill$\Box$ 
    \end{proof}

    \begin{theorem} \label{ global lim inf}
        Under Assumption \ref{assume global 1}, $ \underset{k \rightarrow \infty}{\liminf} \| \grad f(x_k)\| = 0$. 
    \end{theorem}
    
    \begin{proof} 
      Note that $s_k = 0_{x_k}$ \tred{if and only if} $\text{grad}f(x_k) = 0_{x_k} $.
        Moreover, notice that, for $k(j),k(j+1)\in \mathcal{S}$, we have
        \begin{equation}
          x_{\ell}=x_{k(j+1)} \ \text{ for all } \ell \text{ such that } k(j)+1\le \ell \le k(j+1).
          \label{eq:0516-1}
        \end{equation}
      We prove the assertion by considering two cases: (i) $ |\mathcal{S} | = \infty$ and (ii) $ |\mathcal{S} | < \infty$. 

    \tred{{\bf Proof for the case~(i):}} We consider the case where $|\mathcal{S}| = \infty$. For arbitrary $j \in \{0,1,2,\dots\}$, it holds that 
    \begin{alignat}{2}
     & \; \; \; \; \; f(x_{k(j+1)})-f(x_{k(j)}) \notag \\
     &= f(x_{k(j)+1})-f(x_{k(j)})&\quad &\tred{(\text{by }\eqref{eq:0516-1})}\notag \\
     &\leq - \frac{\eta}{2}\left(\theta^{k(j)}(0)-\theta^{k(j)}(s_{k(j)})\right) 
     & \quad &  ( \text{by } \eqref{ def of rho_k} \text{ and } k(j) \in \mathcal{S})\notag \\
     &\leq - \frac{\eta}{2} \frac{\| \text{grad}f(x_{k(j)}) \|^2}{\| J_{k(j)} \|^2 + \lambda_{k(j)}}
     &\quad& (\text{by } \eqref{inequ from Cauchy step})\notag \\
     &\leq - \frac{\eta}{2} \frac{\| \text{grad}f(x_{k(j)}) \|^2}{M^2 + \mu_{k(j)}\|F(x_{k(j)})\|^2}
     &\quad& ( \text{by } \text{Assumption }\ref{assume global 1}).\label{ f is decreasing}
     \end{alignat}
    The above implies that $\{f(x_k)\}_{k \in \mathcal{S}}$ is monotonically decreasing. Furthermore, this property and Algorithm~\ref{alg1} lead to the fact that $\{f(x_k)\} $ is monotonically non-increasing. Moreover, by the definition of $f$, $ f(x_{k(j)}) \geq 0$ holds for all $j \in \{0,1,2,\dots\}$.
    Therefore,
\begin{equation}
  f(x_{k(j+1)})-f(x_{k(j)}) \rightarrow 0 \; (j \rightarrow \infty)\label{eq:0516-2}
\end{equation}
  holds.
  In what follows, we will show that 
    \begin{equation}
    \underset{j \rightarrow \infty}{\liminf} \;  \| \text{grad}f(x_{k(j)}) \| = 0.\label{eq:0522-1}
    \end{equation}
\tred{To this end, we further divide the current case\,(i) into two cases}: \tred{(i-a)} $\{\mu_k\} $ is bounded and \tred{(i-b)} $\{\mu_k\}$ is unbounded.     
 %

    \tred{Consider the case~(i-a).
        Let $ \hat{\mu} := \underset{k}{\sup} \mu_k < \infty$.
 By \eqref{ f is decreasing}, we have}
    \begin{eqnarray*}
        f(x_{k(j+1)})-f(x_{k(j)}) &\leq& - \frac{\eta}{2} \frac{\| \text{grad}f(x_{k(j)}) \|^2}{M^2 + \tred{\hat{\mu} }\|F(x_{k(j)})\|^2} \\
        &\leq&  - \frac{\eta}{2} \frac{\| \text{grad}f(x_{k(j)}) \|^2}{M^2 + \tred{\hat{\mu} }\|F(x_0)\|^2},
    \end{eqnarray*}
    where the second inequality follows from $\| F(x_{k(j)}) \| \leq \| F(x_{0}) \| $ by the monotonically non-increasing property of $ \{f(x_k)\}$.
    \tred{Combining this with \eqref{eq:0516-2} yields}
    $ \lim_{j \rightarrow \infty} \| \text{grad}f(x_{k(j)}) \| = 0 $. 
\tred{Thus, \eqref{eq:0522-1} is established in the case\,(i-a).}
    
    \tred{We next consider the case\,(i-b)} where $ \{\mu_k \}$ is unbounded. \tred{From the construction of Algorithm~\ref{alg1} along with the assumptions that $ | \mathcal{S}|=\infty $ and $ \{\mu_k\}$ is unbounded}, it follows that $ \{\mu_{k(j)}\}$ is unbounded.
    Noting $ \| \grad f(x_{k(j)}) \|^2 \geq \lambda_{k(j)}^2 \|s_{k(j)}\|^2 $ from Lemma~\ref{lemm for two inequs},  \eqref{update_lambda} and \eqref{ f is decreasing}, we have 
    \begin{eqnarray}
        f(x_{k(j+1)})-f(x_{k(j)}) &\leq& - \frac{\eta}{2} \frac{ \mu_{k(j)}^2 \| F(x_{k(j)}) \|^4}{M^2 + \mu_{k(j)}\|F(x_{k(j)})\|^2} \| s_{k(j)} \|^2 \notag \\
        &\leq&  - \frac{\eta}{2} \frac{ \mu_{k(j)}^2 \| F(x_{k(j)}) \|^4}{M^2 + \mu_{k(j)}\|F (x_0)\|^2} \| s_{k(j)} \|^2  , \label{ f dec 2}
    \end{eqnarray}
    where the second inequality follows from $\| F(x_{k(j)}) \| \leq \| F(x_{0}) \| $.
       \tred{From \eqref{ f dec 2} and \eqref{eq:0516-2}}, it follows that 
    \begin{eqnarray}
        \lim_{j \rightarrow \infty } \frac{ \mu_{k(j)}^2 \| F(x_{k(j)}) \|^4}{M^2 + \mu_{k(j)}\|F (x_0)\|^2} \| s_{k(j)} \|^2  = 0.\label{lim of right}
    \end{eqnarray}
Since $\{\mu_{k(j)}\}$ is unbounded, there exists a subsequence such that 
$\frac{ \mu_{k(j)}^2}{M^2 + \mu_{k(j)}\|F (x_0)\|^2}$ diverges as $j\to\infty$.
This fact along with \eqref{lim of right} results in
    \begin{eqnarray}
        \liminf_{j \rightarrow \infty} \| F(x_{k(j)}) \|^2 \| s_{k(j)} \| = 0, \label{lim inf F_k s_k}
    \end{eqnarray}
    from which we will derive \eqref{eq:0522-1} below.  Notice that 
    \begin{eqnarray*}
        0 = \liminf_{j \rightarrow \infty} \| F(x_{k(j)}) \|^2 \| s_{k(j)} \| \geq \left( \liminf_{j \rightarrow \infty} \| F(x_{k(j)}) \|^2\right) \left( \liminf_{j \rightarrow \infty} \|s_{k(j)} \|\right).
    \end{eqnarray*}
Suppose $\underset{j \rightarrow \infty}{\liminf} \| F(x_{k(j)}) \|^2 >0 $. We then obtain $\underset{j \rightarrow \infty}{\liminf} \|s_{k(j)} \| = 0.$
    Then, there exists some \tred{$J\subseteq \{0,1,2,\dots\} $} such that 
    $ \underset{j \in J, j \rightarrow \infty} \lim \|s_{k(j)} \| = 0$. Using this and \eqref{linear eq of LM}, we obtain 
    \begin{eqnarray*}
        \lim_{j \in J, j \rightarrow \infty} \| \grad f(x_{k(j)}) \| &=& \lim_{j \in J, j \rightarrow \infty} \left\| \left( J_{k(j)}^* J_{k(j)} + \mu_{k(j)} \| F(x_{k(j)}) \|^2 I_{k(j)} \right) s_{k(j)}  \right\| \\
        & \leq & \lim_{j \in J, j \rightarrow \infty}   \left\| J_{k(j)}^* J_{k(j)} + \mu_{k(j)} \| F(x_{k(j)}) \|^2 I_{k(j)} \right\|   \| s_{k(j)} \| \\
        &\leq& \lim_{j \in J, j \rightarrow \infty} \left(M^2 + \| F(x_0) \|^2 \mu_{k(j)} \right) \|s_{k(j)} \|  \\
        &=& 0, 
    \end{eqnarray*}
where 
$I_{x_{k(j)}}$ denotes the identity mapping on $T_{x_{k(j)}}\mani$ and 
the last equality follows from $ \underset{j \in J, j \rightarrow \infty} \lim \|s_{k(j)} \| = 0$ and Lemma \ref{mu_k s_k is bounded}. Next, suppose $\underset{j \rightarrow \infty}{\liminf} \| F(x_{k(j)}) \|^2 =0 $.
    \tred{Since $\grad f(x_{k(j)}) = J_{k(j)}^* F(x_{k(j)})$ and $\|J_{k(j)}^*\|$ is bounded by Assumption \ref{assume global 1}, \eqref{eq:0522-1} is ensured. 
Consequently, \eqref{eq:0522-1} is established in the case (i)-b.}

Now, combining the cases (i)-a and (i)-b, we gain \eqref{eq:0522-1} in the whole case (i). 
Lastly, by noting \eqref{eq:0516-1} again,
$$\underset{k \rightarrow \infty}{\liminf} \| \grad f(x_{k}) \|
\le \liminf_{j \rightarrow \infty} \| \text{grad}f(x_{k(j)}) \| = 0$$
is ensured.
    
    
     \tred{{\bf Proof for the case~(ii)}: In turn, we} consider the case (ii) where $ |\mathcal{S}| < \infty$.
      For each $ \mu > 0 $, define 
    \begin{align*}
    \theta^k_{\mu}(s) &:= \| F(x_k) + J_k s \|^2 + \mu \| F(x_k) \|^2 \| s \|^2,\\ 
    s_k(\mu) &:= \underset{s \in T_{x_k} \mathcal{M}}{\text{arg min }} \theta^k_\mu(s).
    \end{align*}
    Then, by replacing $ \lambda_k $ with $\mu \| F(x_k) \|^2$ in Lemma~\ref{lemm for two inequs}, we have
    \begin{eqnarray}
              && \|s_k(\mu) \| \leq \frac{\| \grad f(x_k) \|}{\mu \|F(x_k)\|^2} \label{s_k boundef by grad}, \\
        && - \langle \grad f(x_k), s_k(\mu) \rangle \geq \frac{\| \grad f(x_k) \|^2}{\|J_k\|^2 + \mu \|F(x_k)\|^2}. \label{grad sk inner product} 
    \end{eqnarray}
    \tred{To derive a contradiction, we suppose $ \| \grad f( x_{\bar{k}} ) \| \neq 0$.
    Since $\bar{k} := \max_{k \in \mathcal{S}} k <\infty$ by assumption,} 
    all iterations after the $\bar{k}$-th iteration are unsuccessful, \tred{implying that}
    \begin{eqnarray*}
        f(x_{\bar{k}})-f(R_{x_{\bar{k}}}(s_{\bar{k}}(\mu))) < \frac{\eta}{2}(\theta^{\bar{k}}_\mu(0)-\theta^{\bar{k}}_\mu(s_{\bar{k}}(\mu))) \; \text{ for all  }  \; \mu \geq \mu_{\bar{k}} .
    \end{eqnarray*} 
\tred{As it holds that} 
    \begin{eqnarray}
        && \theta^{\bar{k}}_\mu(0)-\theta^{\bar{k}}_\mu(s_{\bar{k}}(\mu)) \notag \\
        &=& -2 \langle \text{grad}f(x_{\bar{k}}),s_{\bar{k}}(\mu) \rangle - \| J_{\bar{k}} s_{\bar{k}}(\mu) \|^2 - \mu \|F (x_{\bar{k}}) \|^2 \|s_{ \bar{k}} (\mu)\|^2 
        \notag \\
        & \leq & -2 \langle \text{grad}f(x_{\bar{k}}),s_{\bar{k}}(\mu) \rangle, \label{diffe of theta}
    \end{eqnarray}
    \tred{where the equality holds in a way analogous to  \eqref{theta_k_s}, we obtain}  
    \begin{eqnarray}\label{global (i)}
       f(x_{\bar{k} })-f(R_{x_{\bar{k} }}(s_{\bar{k} }(\mu))) < - \eta \; \langle \text{grad}f(x_{\bar{k} }),s_{\bar{k} }(\mu) \rangle 
    \end{eqnarray}
    for all $\mu \geq \mu_{\bar{k} }$.
    By \eqref{s_k boundef by grad} with $k = \bar{k}$, 
    \begin{equation}
    \lim_{\mu\to \infty} \|s_{\bar{k} }(\mu)\|=0.\label{eq:0522-2}
    \end{equation}	
     From the $C^1$ property of $f$, Taylor's expansion yields that 
    \begin{eqnarray}
      f(R_{x_{\bar{k} }}(s_{\bar{k} }(\mu))) = f(x_{\bar{k} })+\langle \text{grad}f(x_{\bar{k} }),s_{\bar{k} }(\mu) \rangle + e_\mu \label{star}
    \end{eqnarray}
    with $ e_\mu  = o \left( \| s_{\bar{k} }(\mu) \| \right) $. By combining \eqref{star} with (\ref{global (i)}), we find that
    \begin{eqnarray} \label{ e_mu inequ}
        -e_\mu < (1-\eta)\langle \text{grad}f(x_{\bar{k} }),s_{\bar{k} }(\mu) \rangle
    \end{eqnarray}
    holds. Since $ \eta \in (0,1)$, the right-hand side is negative and thus $e_\mu > 0$ follows, \tred{from which} we have
    \begin{alignat}{2}
        & \; \; \; \; \; \frac{e_\mu}{\|s_{\bar{k} }(\mu)\|} \notag \\
        &> (1-\eta) \frac{\| \grad f(x_{\bar{k} }) \|^2}{\|J_{\bar{k} }\|^2 + \mu \|F (x_{\bar{k} })\|^2} \frac{1}{\| s_{k(\mu)} \|} \notag
        & \quad & (\text{by } \eqref{grad sk inner product} \text{ and } \eqref{ e_mu inequ}) \\
        &\geq  (1-\eta) \frac{\| \grad f(x_{\bar{k} }) \|}{\|J_{\bar{k} }\|^2 + \mu \|F (x_{\bar{k} })\|^2} \mu \| F(x_{\bar{k}}) \|^2  \notag
        &\quad& (\text{by } \eqref{s_k boundef by grad}),
    \end{alignat}
    which is equivalent to 
    \begin{eqnarray}
        \| \grad f(x_{\bar{k} }) \| < \frac{1}{1-\eta}\left( \frac{\|J_{\bar{k} }\|^2 + \mu \|F (x_{\bar{k} })\|^2}{\mu \| F(x_{\bar{k}}) \|^2} \right) \frac{e_\mu}{\|s_{\bar{k} }(\mu)\|}. 
        \notag
    \end{eqnarray}
 \tred{By driving $ \mu \rightarrow \infty$ in the above,}  the right-hand side
 converges to $0$ \tred{since 
 \eqref{eq:0522-2}} and $ e_\mu  = o \left( \| s_{\bar{k} }(\mu) \| \right)$ hold. This contradicts the assumption $\grad f(x_{\bar{k} }) \neq 0_{x_{\bar{k} }} $. As a result, we conclude that $\| \text{grad}f(x_{\bar{k}}) \| = 0$. 
    In this case \tred{(ii)}, \tred{any iteration points never vary} after the $\bar{k}$-th iteration and hence $\lim_{k \rightarrow \infty} \| \text{grad}f(x_k) \| =\| \text{grad}f(x_{\bar{k}}) \| =0$ holds.
    
    \tred{The whole proof is complete.}
    \hfill$\Box$ 
    \end{proof}
    
    \tred{In view of the proof of Theorem \ref{ global lim inf} for the cases (i)-a and (ii),
    we have the following corollary:}	 
    \begin{corollary} \label{coro}
        If $\{\mu_k\} $ is bounded and Assumption \ref{assume global 1} holds, then $$ \underset{k \rightarrow \infty}{\lim} \| \grad f(x_k) \| = 0.$$
    \end{corollary}
    As we will show later, if $\grad f $ is Lipschitz continuous and \eqref{P} is nonzero-residual, then, $\{\mu_k\} $ is ensured to be bounded. \\

    \subsection{Iteration complexity}\label{iteration complexity analysis}
    Next, we analyze the iteration complexity of RLM. For this purpose, we require the Lipschitz continuity of $ \text{grad}f$ as in the Euclidean setting. 
    
       \begin{assume}\label{assume global 3} 
        $ \grad f$ is $ L$-Lipschitz continuous on $ \mathcal{L}(x_0)$ where $\mathcal{L}(x_0)$ is defined as in Assumption \ref{assume global 1}. 
    \end{assume}
    Under this assumption, the following useful lemma holds, where 
  the second-order retraction defined below plays an important role.
  \begin{definition}
    A retraction $R$ is a second-order retraction if and only if for all $ (x,v) \in \TM$, the smooth curve $c(t)$ defined as $ c(t) := R_x (tv)$ has zero acceleration at $t=0$, i.e., $ c^{\prime \prime}(0) = 0$.
\end{definition}  
  For instance, a map known as {exponential map} is a second-order retraction.
    \begin{lemma}\label{descent lemma}
\tred{Suppose that the retraction $R$ is second-order
and Assumption\,\ref{assume global 3} holds.
Then, 
for any $(x,s)$ such that $(x,R_x(s))\in \mathcal{L}(x_0)\times \mathcal{L}(x_0)$,
we have}
\begin{equation}
            f(R_x(s)) \leq f(x)+\langle \grad f(x),s \rangle + \frac{L}{2}\|s\|^2. \label{inequ of Lemma4}
        \end{equation}
    \end{lemma}
\begin{proof}
    \tred{The proof follows from \cite[Exercise 10.56]{Boumal} easily.}
\hfill$\Box$
\end{proof}    
    To use Lemma \ref{descent lemma}, we restrict retractions to second-order ones throughout the analysis of iteration complexity. 
    
    For any positive number $\epsilon $, we define some notations as below:
     \begin{align*}   
        &j_{\epsilon}:=\min\{j\mid \| \text{grad}f(x_j) \| < \epsilon\mbox{ or }
        f(x)<\epsilon\},\\
        &\mathcal{S}_{\epsilon} := \{0,1,\dots,j_{\epsilon}-1\} \cap \mathcal{S},\\
        &\mathcal{U}_{\epsilon} := \{0,1,\dots,j_{\epsilon}-1\} \setminus \mathcal{S}_{\epsilon}.
     \end{align*}
    Our objective is to evaluate the worst-case iteration number which is needed to reach a point $x$ such that $ \| \text{grad}f(x) \| < \epsilon$ or $f(x) <\epsilon $ hold.     
\tred{
In other words, we wish to evaluate $j_{\epsilon}$ $(=|\mathcal{S}_{\epsilon}|+|\mathcal{U}_{\epsilon}|)$ in the worst case. The analysis will be conducted by tracing the following three steps one by one:  
}
   \begin{enumerate}
        \item[(1)]  We give a sufficient condition for the $k$-th iteration to be successful.
        \item[(2)] We derive an upper bound of $ | \mathcal{S}_{\epsilon} |$. 
        \item[(3)] 
        \tred{We evaluate the maximum number of unsuccessful iterations 
        occurring consecutively, and then give an upper bound of $ | \mathcal{U}_{\epsilon} |$.}
        \end{enumerate}
Hereinafter, we denote 
\begin{eqnarray}
        \kappa := \frac{\frac{L}{2}+\sqrt{\frac{L^2}{4}+2(1-\eta)LM^2}}{2(1-\eta)}. \label{defi of kappa}
\end{eqnarray}
    For the above step (1), we introduce the following lemma.
    \begin{lemma}\label{lambda bounded}
        \Add{Under Assumptions~\ref{assume global 1} and~\ref{assume global 3}}, if $ \lambda_k \geq \tred{\kappa} $, then the $k$-th iteration is successful.
    \end{lemma}   
    \begin{proof}
        We have
    \begin{alignat}{2}
     & \; \; \; \; f(x_k) - f(R_{x_k} (s_k)) - \frac{\eta}{2}(\theta^k(0)-\theta^k(s_k)) \notag \\
     & \geq  - \langle \text{grad}f(x_k),s_k \rangle -\frac{L}{2}\|s_k\|^2 - \frac{\eta}{2}(\theta^k(0)-\theta^k(s_k)) \notag
     &\quad& (\text{by } \eqref{inequ of Lemma4} \text{ in Lemma~\ref{descent lemma}})\\
     & \geq  - \langle \text{grad}f(x_k),s_k \rangle -\frac{L}{2}\|s_k\|^2 + \eta \; \langle \text{grad}f(x_k),s_k \rangle \notag 
     &\quad&  (\text{in a manner similar to } \eqref{diffe of theta}) \\
     & \geq  (1-\eta) \langle \text{grad}f(x_k), -s_k \rangle - \frac{L}{2} \|s_k \|^2 \notag \\
     & \geq  (1-\eta)\frac{\| \text{grad}f(x_k) \|^2}{ \|J_k \|^2 + \lambda_k} -\frac{L}{2 \lambda_k^2}\| \text{grad}f(x_k) \|^2 \notag 
     &\quad& (\text{by } \eqref{s_k grad rela} \text{ and } \eqref{inner prod of grad s_k})  \\
     & \geq  (1-\eta)\frac{\| \text{grad}f(x_k) \|^2}{M^2 + \lambda_k} -\frac{L}{2 \lambda_k^2}\| \text{grad}f(x_k) \|^2 \notag
     &\quad&  ( \text{by Assumption~\ref{assume global 1}})  \\
     &= \left( \frac{1-\eta}{M^2+\lambda_k} - \frac{L}{2\lambda_k^2}\right) \| \text{grad}f(x_k) \|^2. \notag
    \end{alignat}
    Therefore, if $ \frac{1-\eta}{M^2+\lambda_k} - \frac{L}{2\lambda_k^2} \geq 0 $ holds, then the $k$-th iteration is successful.
    Lastly, since 
     $\frac{1-\eta}{M^2+\lambda_k} - \frac{L}{2\lambda_k^2} \geq 0$ 
     is equivalent to 
     $(1-\eta)\lambda_k^2 -\frac{L}{2}\lambda_k - \frac{L M^2}{2} \geq 0$,
    we conclude that if $\lambda_k \geq \frac{\frac{L}{2}+\sqrt{\frac{L^2}{4}+2(1-\eta)LM^2}}{2(1-\eta)}\tred{=\kappa}$, then the $k $-th iteration is successful.
    \hfill $ \Box$
    \end{proof}
    
    \tred{Recall that the parameter $\beta>1$ is set in Algorithm~\ref{alg1}.}
 Using Lemma\,\ref{lambda bounded}, we can show that $\{\mu_k\}_{0 \leq k \leq j_{\epsilon} } $ is bounded by 
   $$\mu_{\max}(\epsilon) := \frac{\beta \kappa}{\text{\Add{$2$}}\epsilon},$$
 namely, it holds that 
        \begin{equation}
        \max_{0\leq k\leq j_{\epsilon}}\mu_k\le \mu_{\max}(\epsilon). \label{eq:0522-3}
        \end{equation}
\tred{Indeed}, $\lambda_k (= \mu_k \| F(x_k) \|^2)$ is bounded by $\beta \kappa $ because of Lemma\,\ref{lambda bounded} and thus,
    \begin{eqnarray*}
        \mu_k < \frac{\beta \kappa}{\| F(x_k) \|^2} \leq \frac{\beta \kappa}{\text{\Add{$2$}}\epsilon}=\tred{\mu_{\max}(\epsilon)},
    \end{eqnarray*}
    where the second inequality follows from $ \tred{f(x_k)=\frac{1}{2}}\| F(x_k) \|^2 \geq \epsilon$ for all $k=0,\dots,j_{\epsilon}-1$. 
    \begin{remark}\label{boundMu}
When Assumptions~\ref{assume global 1} and \ref{assume global 3} hold and the global optimal value of \eqref{P} is positive, then $ \{\mu_k\}$ is bounded by $ \frac{\beta \kappa}{2 \underset{x \in \mani} \min f(x)}$.
    \end{remark}
 
As the next step~(2), we give an upper-bound of $ |\mathcal{S}_{\epsilon}|$ specifically in the following lemma.
    \begin{lemma}
        Under Assumptions~\ref{assume global 1} and~\ref{assume global 3},
    $
     |\mathcal{S}_{\epsilon}| \leq 2 f(x_0)\frac{M^2+\mu_{\max}(\epsilon)\|F(x_0)\|^2}{\eta} \epsilon^{-2}
    $
    holds.
    \end{lemma}
    
    \begin{proof} 
        First, for all $ j \in \{0,1,\dots,|\mathcal{S}_{\epsilon}|-1\}$, \tred{the inequality\,\eqref{ f is decreasing} is obtained in a similar manner to 
        Theorem\,\ref{ global lim inf}.}
        \tred{Then, } 
        $\mu_k \|F(x_k) \|^2 \leq \mu_{\max}(\epsilon) \|F(x_0) \|^2 $ holds 
        and thus, we have 
    \begin{eqnarray*}
     f(x_{k(j)})-f(x_{k(j+1)}) \geq \frac{\eta}{2} \frac{1}{M^2 + \mu_{\max}(\epsilon)\|F(x_0)\|^2} \| \text{grad}f(x_{k(j)}) \|^2.
    \end{eqnarray*}
    By summing up the above inequality from $j=0 $ to $|\mathcal{S}_{\epsilon}|-1$
    \tred{and noting $f(x_{k(| \mathcal{S}_{\epsilon}|)})\ge 0$}, we obtain
    \begin{eqnarray*}
     f(x_0) &=& f(x_{k(0)}) \\
          & \geq & \frac{\eta}{2 \left( M^2 + \mu_{\max}(\epsilon)\|F(x_0)\|^2 \right)} \sum_{j=0}^{|\mathcal{S}_{\epsilon}|-1}\| \text{grad}f(x_{k(j)}) \|^2 +f(x_{k(| \mathcal{S}_{\epsilon}|)})\\
     & \geq & \frac{\eta}{2 \left( M^2 + \mu_{\max}(\epsilon)\|F(x_0)\|^2 \right)} \sum_{j=0}^{|\mathcal{S}_{\epsilon}|-1}\| \text{grad}f(x_{k(j)}) \|^2 \\
     &\geq& \frac{\eta \epsilon^2}{ 2 \left( M^2 + \mu_{\max}(\epsilon)\|F(x_0)\|^2 \right) } |\mathcal{S}_{\epsilon}|,
    \end{eqnarray*}
    \tred{where 
    the last inequality follows from the assumption
    $\| \grad f(x_{k(j)})\| \geq \epsilon$ for $j=0,1,\ldots,|\mathcal{S}_{\epsilon}|-1$.}
    Consequently, we ensure 
    \begin{eqnarray*}
    |\mathcal{S}_{\epsilon}| \leq 2 f(x_0)\frac{M^2+\mu_{\max}(\epsilon)\|F(x_0)\|^2}{\eta} \epsilon^{-2} .
    \end{eqnarray*}
    \hfill$\Box$
    \end{proof}
    As the final step~(3), we prove the following lemma.
    \begin{lemma}
        \Add{Suppose that Assumptions~\ref{assume global 1} and~\ref{assume global 3} hold. Then,
    $
     | \mathcal{U}_\epsilon| \leq c_{\max}(\epsilon)|\mathcal{S}_{\epsilon}|
    $}
    holds where $ c_{\max}(\epsilon) := \lceil \log_{\beta}{\left(\frac{\kappa}{\text{\Add{$2$}}\mu_{\min}}\epsilon^{-1}\right)}\rceil $. \Add{Here, $ \mu_{\min}$ is the constant prefixed in Algorithm~\ref{alg1} and $\lceil\cdot\rceil$ is the ceiling function.}\\
    \end{lemma}
    
    \begin{proof} 
\tred{
Recall $\beta>1$.
Since
$\| F(x_k) \|^2 \geq \text{\Add{2}}\epsilon$ holds for an arbitrarily chosen $k < j_{\epsilon}$, we have}
    \begin{eqnarray*}
     \tred{\beta^{c_{\max}(\epsilon)}\mu_{\min} \|F(x_k) \|^2\geq} \beta^{\log_{\beta}{\left(\frac{\kappa}{\text{\Add{$2$}}\mu_{\min}}\epsilon^{-1}\right)}} \mu_{\min} \|F(x_k) \|^2 = \kappa \frac{\|F(x_k) \|^2}{\text{\Add{$2$}}\epsilon} \geq \kappa.
    \end{eqnarray*}
\tred{Hence, 
if the $k$-th iteration is right after consecutive $c_{\max}(\epsilon)$ unsuccessful iterations, the assumptions of Lemma~\ref{lambda bounded} are fulfilled because
$\lambda_k\ge \beta^{c_{\max}(\epsilon)}\mu_{\min}$.
Therefore, the $k$-th iteration is successful.  
This implies that the number of consecutive unsuccessful iterations is at most $ c_{\max}(\epsilon)$.}
    

Now we can upper-bound $|\mathcal{U}_{\epsilon} | $ by $ c_{\max}(\epsilon)|\mathcal{S}_{\epsilon}| $, because 
there occur alternately at most
$c_{\max}(\epsilon) $ unsuccessful iterations and one successful iteration  until 
\tred{the number of iterations} reaches $ j_{\epsilon}$.
    
    \hfill$\Box$
    \end{proof}
    Finally, we obtain the following result about the iteration complexity of Algorithm \ref{alg1}.
    \begin{theorem}\label{iter comp} 
        Under Assumptions \ref{assume global 1} and \ref{assume global 3}, the iteration complexity of Algorithm \ref{alg1} to find a solution satisfying $\| \grad f(x) \| < \epsilon $ or \Add{$ f(x)<\epsilon$} is bounded by $O\left( \log{(\epsilon^{-1})}\epsilon^{-3}\right) $.
    \end{theorem}
    
    \begin{proof}
        The number $j_{\epsilon} $ is bounded from above as follows:
    \begin{eqnarray}
     j_{\epsilon} &=& | \mathcal{S}_{\epsilon} | + |\mathcal{U}_{\epsilon}| \notag \\
     & \leq & (1+c_{\max}(\epsilon)) | \mathcal{S}_{\epsilon} |  \notag \\
     & \leq & (1+c_{\max}(\epsilon)) \left(2 f(x_0)\frac{M^2+\mu_{\max}(\epsilon)\|F(x_0)\|^2}{\eta} \epsilon^{-2}\right). \label{ j_e bound}
     \end{eqnarray}
    
     From \eqref{ j_e bound} and the definitions of $\mu_{\max}(\epsilon) $ and $c_{\max}(\epsilon)$, the assertion follows.
     \hfill$\Box$
    \end{proof}

\section{Analysis on local convergence} \label{chap:local}
In this section, we show the local convergence properties of Algorithm~\ref{alg1}, by dividing it into (1) zero-residual and (2) nonzero-residual cases. First, we study the local convergence behavior of the algorithm around a zero-residual stationary point, namely, $x^*$ such that $f(x^*) = 0$, implying that $x^*$ is a stationary point since it is optimal.
Second, we analyze the behavior around a solution $x^*$ which is a stationary point but $f(x^*)\neq 0$.

    \subsection{Notations for local convergence analysis}\label{sec:notations}
    \tred{We first introduce additional} notations. Let $X^*$ denote the set of stationary points with the same residual $f^*:=f(x^*)$ as $x^*$, i.e., 
    $$
    \tred{X^{\ast}:=\{x\in \mani\mid \gradf(x)=0,\ f(x) = f^*\}. }
    $$
    Given $x \in \mani$, we define the distance between $x$ and $X^*$ as 
    \begin{eqnarray*}
      \Dist(x,X^*) := & \underset{\hat{x} \in X^*}{\min} \dist(x,\hat{x}).
  \end{eqnarray*}
  Moreover, we write $ \bar{x}$ to denote a point which is the closest to $x$ in $X^*$, that is,
  \begin{eqnarray*}
    \bar{x} \in &  \underset{\hat{x} \in X^*}{\text{ arg min }} \dist(x,\hat{x}).
  \end{eqnarray*}
Hereinafter, we often use $\bxk$ defined by setting $x:=x_k$ above.
Let $B(x,b) \subset \mani$ be the ball with radius $b$ centered at $x$, i.e., $ B(x,b) := \{ y \in \mani \; | \; \dist(x,y) \leq b\}$.
Note that the Jacobian matrix $J$ is ensured to be bounded over $B(x^*,b)$ without any specific assumptions. This is due to the $C^1$ property of $F$ and the compactness of $B(x^*,b)$. Let $K$ denote the upper bound of the operator norm of $J$. Namely,
\begin{equation}  \label{upperK}
  \|J(x)\| \leq K 
\end{equation}
holds for all $x \in B(x^*,b)$.

\subsection{Basic assumptions and lemmas}
In this subsection, we give common assumptions and lemmas, which are used 
throughout the analysis for zero- and nonzero-residual cases. 

From the inverse function theorem, there exists an open set $U$ of $T_x \mathcal{M}$ containing $0_x$ such that $R_x : U \rightarrow R_x(U) \; (\subseteq \mani) $ is \tred{a diffeomorphism}. Let $R_x^{-1}:R_x(U) \rightarrow U$ be \tred{the} inverse function. 
\begin{assume}
     \label{Assumption for local zero}
        The stationary point $ x^*\in X^*$ satisfies the following conditions: 
        \begin{enumerate}
        \item[{\rm (a)}] 
    There exist $b \in (0,\infty)$ and $ c_1 \in (0,\infty)$ such that $ \| J(y)R^{-1}_y(x) - (F(x)-F(y)) \| \leq c_1 \| R^{-1}_y(x) \|^2 $ holds for all $  x,y \in B(x^*,b) $. 
    \item[{\rm (b)}] There exists $ c_2 >0 $ such that $ c_2 \| R^{-1}_x(\bar{x}) \| \leq \| F(x)-F(\bar{x}) \| $ holds for all $ x \in B(x^*,b)$.
    In particular, in the zero-residual case, i.e., $F(\bar{x})=0$, the inequality is reduced to $c_2 \| R^{-1}_x(\bar{x}) \| \leq \| F(x) \|$.
    \item[{\rm (c)}] $ \{\mu_k \}$ is upper-bounded by \tred{some positive constant, say} $\tred{\mumaxz}$.
    \end{enumerate}
    \end{assume}
The first and second assumptions are often made in the local convergence analysis for the Euclidean LM method. Indeed, they correspond to  
\cite[Assumption 2.1 (a) and (b)]{Yamashita}, respectively. In particular, the second one  is often referred to as the local error-bound condition in many articles regarding the Euclidean LM following \cite{Yamashita}. This condition is weaker than the injectiveness of the Jacobian matrix supposed in the local analysis for the RGN\,\cite{RGN}.
    
    By taking the constant $b$ in the above assumption to be sufficiently small, the following relation \eqref{eq:0605-1} is ensured under the above assumptions. 
\begin{equation}
    B(x^*,b) \subset R(B^{T_{x^*} \mani}(\injR (x^*))), \label{eq:0605-1}
    \end{equation}
where
    $ \injR (x^*)$ denotes the injectivity radius of $R$ at $x^*$ defined formally as follows: 
    \begin{definition}\label{inj of exp}
        The injectivity radius of $R$ at $x \in \mani$, denoted by $\injR (x)$, is the supremum over radii $r>0$ such that $R(x)$ is a diffeomorphism on the open ball 
        \begin{eqnarray*}
            B^{T_x \mani }(r) := \{ v \in T_{x} \mani \; | \; \|v\|_x < r \}.
        \end{eqnarray*}
    \end{definition}
Hereinafter, we additionally suppose \eqref{eq:0605-1} holds.
    
     \tred{Since} any bounded and closed set is compact in a complete Riemannian manifold, so is $B(x^*,b)$. Hence,  
 there exist the minimum and maximum eigenvalues of the matrix of Riemannian metric $\langle \cdot, \cdot \rangle$ on $ B(x^*,b)$, denoted by $ \lambda_{\min}$ and $ \lambda_{\max}$, respectively. 
%
Define the following constant $c$ in terms of $\lambda_{\min}$ and $ \lambda_{\max}$:  
    \begin{eqnarray*}
        c := \sqrt{\frac{\lambda_{\max}}{\lambda_{\min}}}\ge 1.
    \end{eqnarray*}
	    \begin{lemma} \label{lemma inv retraction}
The following holds:
    \begin{eqnarray}
    \tred{c^{-1}}\dist (x,y) \leq \| R^{-1}_x(y) \|_x \leq \tred{c}\,\dist (x,y) \; \text{ for all } x,y \in B(x^*,b). \label{dist and inv retraction}
    \end{eqnarray}
    \end{lemma}
    
    \begin{proof} 
    Note that $R_x(U)$ is an open neighborhood of $x$ and $U$ can be identified with an open set of $\mathbb{R}^d$. Therefore, for $x \in \mathcal{M}$, $(R_x(U), R_x^{-1})$ forms a coordinate neighborhood around $x$.  Hence, we have
    \begin{equation}
    \widehat{R_x^{-1}(y)}=R_x^{-1}(y)\label{eq:1731-1}
    \end{equation}
    for $y \in B(x^*,b)$.
    
    \tred{Choose $ v\in T_x \mathcal{M}$ with $x \in B(x^*,b)$ arbitrarily}.
    By the definition of $\lambda_{\max} $ and $\lambda_{\min}$, we have
    \begin{eqnarray*}
     && \lambda_{\min} \| \hat{v} \|^2 \leq \langle v,v \rangle_x \leq \lambda_{\max} \| \hat{v} \| ^2,
     \end{eqnarray*}
\tblue{where the norm in the left and right sides stands for the Euclidean norm.}
This yields
   \begin{eqnarray*}
   \frac{\| v \|_x}{\sqrt{\lambda_{\max}}} \leq \| \hat{v} \| \leq \frac{\| v \|_x}{\sqrt{\lambda_{\min}}}.
    \end{eqnarray*}
    Substituting $v = R^{-1}_x(y) $ into the above and noting \eqref{eq:1731-1}, we obtain
    \begin{eqnarray}
     \frac{\| R^{-1}_x(y) \|_x}{\sqrt{\lambda_{\max}}} \leq \|{R^{-1}_x(y)}\| \leq \frac{\| R^{-1}_x(y) \|_x}{\sqrt{\lambda_{\min}}}. \label{E main (i)}    \end{eqnarray}
    
    Given the formulation of coordinate neighborhood, the following holds:
    \begin{eqnarray}
     \tblue{\|R_x^{-1}(y) \|} &=& \| R_x^{-1}(x)-R_x^{-1}(y) \| = \| \hat{x}-\hat{y} \|. \label{E main (ii)}
    \end{eqnarray}
Using the inequalities $ \sqrt{\lambda_{\min}} \| \hat{x}-\hat{y} \| \leq \text{dist}(x,y) \leq \sqrt{\lambda_{\max}} \| \hat{x}-\hat{y} \|$ in \cite{riemanndistandlocalcoordinate} for the current local coordinate system, we obtain 
\begin{eqnarray}
    \sqrt{\lambda_{\min}} \| R_x^{-1}(y) \| \leq \dist(x,y) \leq  \sqrt{\lambda_{\max}} \| R_x^{-1}(y) \|, \label{E main (iii)}
\end{eqnarray}
where \eqref{E main (ii)} is used in the place of $\| \hat{x}-\hat{y} \| $.
Noting that 
\begin{eqnarray*}
    \begin{cases}
        \sqrt{\frac{\lambda_{\min}}{\lambda_{\max}}} \| R_x^{-1}(y) \|_x \leq \sqrt{\lambda_{\min}} \| R_x^{-1}(y) \| & \\
        \sqrt{\lambda_{\max}} \| R_x^{-1}(y) \| \leq \sqrt{\frac{\lambda_{\max}}{\lambda_{\min}}} \| R_x^{-1}(y) \|_x
    \end{cases}
\end{eqnarray*}
holds by \eqref{E main (i)}, we find that \eqref{E main (iii)} leads to 
    \begin{eqnarray*}
     & \sqrt{ \frac{\lambda_{\min}}{\lambda_{\max}} } \| R_x^{-1}(y) \|_x \leq \text{dist}(x,y) \leq \sqrt{ \frac{\lambda_{\max}}{\lambda_{\min}} } \| R_x^{-1}(y) \|_x,&
     \end{eqnarray*}
     \tred{equivalently,}
     \begin{eqnarray*}
     &\sqrt{ \frac{\lambda_{\min}}{\lambda_{\max}} } \text{dist}(x,y) \leq \| R_x^{-1}(y) \|_x \leq \sqrt{ \frac{\lambda_{\max}}{\lambda_{\min}} } \text{dist}(x,y).&
    \end{eqnarray*}
      \hfill$\Box$
    \end{proof}

    \begin{lemma}\label{local Lipschitz like}
        Suppose that Assumption \ref{Assumption for local zero}(a) holds. Then, there exists some $ L >0 $ such that $ \| F(x)-F(y) \| \leq L \; \| R^{-1}_y(x) \|$ holds for all $ x,y \in B(x^*,b)$.
    \end{lemma}
    
    \begin{proof} 
        For $ x,y \in B(x^*,b)$, we have
    \begin{eqnarray*}
     \| F(x)-F(y) \| &\leq& \| (F(x)-F(y))- J(y) R^{-1}_y(x) \| + \| J(y) R^{-1}_y(x) \| \\
     & \leq & c_1 \| R_y^{-1}(x)\|^2 +\sqrt{\sum_{i=1}^{m} \| \text{grad}F_i(y) \|_y^2} \; \| R^{-1}_y(x)\| \\
     & = & \left(c_1 \| R^{-1}_y(x) \| + \sqrt{\sum_{i=1}^{m} \| \text{grad}F_i(y) \|_y^2} \; \right) \; \| R^{-1}_y(x) \| \\
     & \leq &\left(2c_1 cb + \sqrt{\sum_{i=1}^{m} \| \text{grad}F_i(y) \|_y^2} \; \right) \; \| R^{-1}_y(x) \|, 
    \end{eqnarray*}
    where Assumption~\ref{Assumption for local zero} (a) is applied in the second inequality and the final one follows from $\| R^{-1}_y(x) \| \leq c \;\text{dist}(x,y) \leq 2cb$.
    
    Here, $\text{grad}F_i \; (1 \leq i \leq m)$ are continuous because $F$ is a $C^1$ function. Moreover, the norm defined by the Riemannian metric is continuous. Consequently, $\sqrt{\sum_{i=1}^{m} \| \text{grad}F_i(y) \|_y^2}$ is a continuous function on the compact set $ B(x^*,b)$ and hence it attains the maximum value on $B(x^*,b)$.
    
      In terms of 
$$
L := 2c_1 cb + \max_{z \in B(x^*,b)}{\sqrt{\sum_{i=1}^{m} \| \text{grad}g_i(z) \|_z^2}}>0,
$$
the following inequality is established:
    \begin{eqnarray*}
     \| F(x)-F(y) \| \leq L \; \| R^{-1}_y(x) \| \;  \text{ for all } x,y \in B(x^*,b).
    \end{eqnarray*}
    \hfill$\Box$
    \end{proof} 

 \subsection{Quadratic convergence for zero-residual cases} \label{local zero}
In this subsection, 
we consider the zero-residual case, namely, $f^*:=f(x^*)=0$ for the stationary point $x^*$. 
    \tblue{We suppose that $\text{flag}^{\text{nz}} = \text{false} $ in Algorithm~\ref{alg1}}.    

    \begin{lemma}\label{Lemma 5.1.3}
        Suppose $x_k \in B(x^*,\frac{b}{2})$ with some $k$. Under Assumption \ref{Assumption for local zero}, we have
    \begin{eqnarray}
     && \| s_k \| \leq c_3 \| R^{-1}_{x_k}(\overline{x_k}) \|, \label{Lemma 5.1.3-1} \\
     && \| J_k s_k+F(x_k)\| \leq c_4 \| R^{-1}_{x_k}(\overline{x_k}) \|^2, \label{Lemma 5.1.3-2}
    \end{eqnarray}
    where
    $$ c_3 := \sqrt{\frac{c_1^2+c_2^2 \; \mu_{\min}}{c_2^2 \; \mu_{\min}}},\ c_4 := \sqrt{c_1^2+L^2 \tred{\mumaxz}}.$$
    \end{lemma}
    
    \begin{proof}
    It follows that $ \theta^k(s_k) \leq \theta^k( R^{-1}_{x_k}(\overline{x_k}))$ from \eqref{ s_k = argmin theta}. 
Moreover, $\lambda_k \| s_k \|_{x_k}^2 \leq \theta^k(s_k)$ holds from \eqref{original theta_k_s}. Therefore, we have  
    \begin{eqnarray}\label{Lemma 5.1.3 (i)}
     \|s_k \|^2 \leq \frac{1}{ \lambda_k} \theta^k(s_k)  &\leq&  \frac{1}{ \lambda_k} \theta^k( R^{-1}_{x_k}(\overline{x_k})) \notag \\
     &=& \frac{1}{ \lambda_k}( \|F(x_k)+ J_k R^{-1}_{x_k}(\overline{x_k}) \|^2 + \lambda_k \| R^{-1}_{x_k}(\overline{x_k}) \|^2).
    \end{eqnarray}
    
    Next, we can ensure $\overline{x_k} \in B(x^*,b)$ under $x_k \in B(x^*,\frac{b}{2})$ as follows:
    \begin{eqnarray*}
     \text{dist}( \overline{x_k},x^*) \leq \text{dist}( \overline{x_k},x_k) + \text{dist}(x_k,x^*) \leq 2 \; \text{dist}(x_k,x^*) \leq b.
    \end{eqnarray*}
    
    Thus, applying Assumption \ref{Assumption for local zero} and noting $F(\overline{x_k})= \bm{0}$ from $\overline{x_k} \in X^*$, we obtain
    \begin{eqnarray} 
     \| F(x_k)+J_k R^{-1}_{x_k}(\overline{x_k}) \|^2 &=& \| J_k R^{-1}_{x_k}(\overline{x_k}) - (F(\overline{x_k})-F(x_k)) \|^2 \notag \\
     &\leq& c_1^2 \| R^{-1}_{x_k}(\overline{x_k}) \|^4, \label{uoaaa}
     \end{eqnarray}
    and 
      moreover, by $c_2 \| R^{-1}_{x_k}(\overline{x_k})\| \leq \| F(x_k) \|$ together with $\lambda_k = \mu_k \| F(x_k) \|^2$ and $ \mu_k \geq \mu_{\min}$, we gain 
     \begin{eqnarray}
     & c_2^2 \; \| R^{-1}_{x_k}(\overline{x_k})\|^2 \leq \frac{\lambda_k}{\mu_{\min}}.\notag
     \end{eqnarray}
Combining these relations with  (\ref{Lemma 5.1.3 (i)}) yields
    \begin{eqnarray*}
     \| s_k \|^2 &\leq& \left( \frac{c_1^2}{\lambda_k} \| R^{-1}_{x_k}(\overline{x_k})\|^2 + 1 \right) \; \| R^{-1}_{x_k}(\overline{x_k})\|^2\\ 
                    &\leq& \left(\frac{c_1^2}{\lambda_k}\frac{\lambda_k}{c_2^2 \; \mu_{\min}} + 1 \right) \; \| R^{-1}_{x_k}(\overline{x_k})\|^2\\
                    &=&\tred{c_3^2\| R^{-1}_{x_k}(\overline{x_k})\|^2}.
    \end{eqnarray*}
    Therefore,
    $
     \| s_k \| \leq c_3 \; \| R^{-1}_{x_k}(\overline{x_k})\|
    $
    has been proved.\\
    
    Next, we show that $\| J_k s_k + F(x_k) \| \leq c_4 \; \| R^{-1}_{x_k}(\overline{x_k})\|^2 $.
    \begin{eqnarray}\label{Lemma 5.1.3 circled 1 }
     \| J_k s_k + F(x_k) \|^2 &\leq& \theta^k(s_k) \; (\text{by } \eqref{original theta_k_s}) \notag \\
      &\leq& \theta^k( R^{-1}_{x_k}(\overline{x_k})) \; (\text{by } \eqref{ s_k = argmin theta}) \notag \\
      &\leq& c_1^2 \| R^{-1}_{x_k}(\overline{x_k})\|^4+\lambda_k \| R^{-1}_{x_k}(\overline{x_k})\|^2, 
    \end{eqnarray}
    where the last inequality follows from \eqref{original theta_k_s} and \eqref{uoaaa}.
    Furthermore, from Lemma \ref{local Lipschitz like},
    it follows that $ \sqrt{\lambda_k} = \sqrt{\mu_k} \| F(x_k) \| = \sqrt{\mu_k}\| F(\overline{x_k})-F(x_k) \| \leq L \sqrt{\mu_k}\; \| R^{-1}_{x_k}(\overline{x_k})\|$. Hence, by \tblue{Assumption \ref{Assumption for local zero} (c)}, 
    \begin{eqnarray}\label{ Lemma 5.1.3 circled 2}
     \lambda_k \leq L^2 \mu_k \; \| R^{-1}_{x_k}(\overline{x_k})\|^2 \leq  L^2 \tred{\mumaxz} \; \| R^{-1}_{x_k}(\overline{x_k})\|^2 
    \end{eqnarray}
    Applying (\ref{ Lemma 5.1.3 circled 2}) to (\ref{Lemma 5.1.3 circled 1 }), we conclude 
    $
     \| J_k s_k + F(x_k) \|^2 \leq (c_1^2+L^2\tred{\mumaxz}) \| R^{-1}_{x_k}(\overline{x_k})\|^4, 
    $
    which is equivalent to 
    \begin{eqnarray*}
      \| J_k s_k + F(x_k) \| \leq c_4 \; \| R^{-1}_{x_k}(\overline{x_k})\|^2 .
    \end{eqnarray*}
    The proof is complete.
    \hfill$\Box$
    \end{proof}    
    Let 
    \begin{eqnarray} \label{c*_def}
        c^* := \frac{-c c_2 c_4 + \sqrt{c^2c_2^2c_4^2 + 4c^2\mumaxz c_2^2 c_3^2 L^2}}{2c^2 \mumaxz c_3^2 L^2} .
    \end{eqnarray}
    Note that $c^*$ is a strictly positive constant.
    \begin{lemma} \label{ for large k, kepp to be successful}
        \Add{Suppose that Assumption~\ref{Assumption for local zero} holds.}
        Moreover, assume that $b$ satisfies $b \leq \frac{2c_2}{c c_4} $.
        Then, there exists some $r^* >0 $ such that if $\{x_k\} $ satisfies $ x_k \in B(x^*,\min{\{\frac{b}{2},r^*\}})$ and $ \dist(x_k,x^*) < c^*$ for all $k \in \{0,1,2,\dots\}$, then all iterations are successful.
    \end{lemma}
    \begin{proof} 
        Let $x_k \in B(x^*,\frac{b}{2})$. 
        From the definition~\eqref{ def of rho_k}, we have 
    \begin{eqnarray}
     1- \rho_k &=& \frac{\frac{1}{2}\theta^k(0)-f(x_k)+f(R_{x_k}(s_k))-\frac{1}{2}\theta^k(s_k)}{\frac{1}{2}(\theta^k(0)-\theta^k(s_k))} \notag \\
     & = & \frac{f(R_{x_k}(s_k))-\frac{1}{2}\theta^k(s_k)}{\frac{1}{2}(\theta^k(0)-\theta^k(s_k))},  \label{ 1- rho_k}
    \end{eqnarray}
    where the second equality follows from $\frac{1}{2}\theta^k(0)= f(x_k)$.
    Before evaluating the denominator in \eqref{ 1- rho_k}, we first show that $\|F(x_k)\|-\|F(x_k)+J_ks_k\|$ is ensured to be positive under the assumption that $b \leq \frac{2c_2}{c c_4} $. This can be verified as follows:
    \begin{alignat}{2}
        &\; \; \; \; \; \|F(x_k)\|-\|F(x_k)+J_ks_k\| \notag \\
        &\geq  c_2 \| R^{-1}_{x_k}(\overline{x_k}) \| -\|F(x_k)+J_ks_k\| \notag 
        &\quad& (\text{by Assumption~\ref{Assumption for local zero}(b)}) \\
        &\geq \| R^{-1}_{x_k}(\overline{x_k}) \| (c_2 - c_4 \| R^{-1}_{x_k}(\overline{x_k}) \| ) \notag
        &\quad& (\text{by \eqref{Lemma 5.1.3-2} in Lemma~\ref{Lemma 5.1.3}} ) \\
        &\geq \| R^{-1}_{x_k}(\overline{x_k}) \| (c_2 - c c_4 \dist(x_k,\overline{x_k}) ) \notag 
        &\quad& (\text{by Lemma~\ref{lemma inv retraction}}) \\
        &\geq \| R^{-1}_{x_k}(\overline{x_k}) \| (c_2 - c c_4 \dist(x_k, x^*) ) \notag
        &\quad& (\text{by the definition of } \overline{x_k}) \\
        &\geq \| R^{-1}_{x_k}(\overline{x_k}) \| (c_2 - \frac{bcc_4}{2} ) \notag
        &\quad& (\text{by the assumption }x_k \in B(x^*,\min{\{\frac{b}{2},r^*\}})) \\
        &\geq 0. \label{positibity of diffe in cur and nxt}
        &\quad& (\text{by the assumption } b \leq \frac{2c_2}{c c_4})
    \end{alignat}
    Then, by Lemma \ref{Lemma 5.1.3}, the denominator in \eqref{ 1- rho_k} is evaluated as
    \begin{eqnarray}
        && \frac{1}{2}\left(\theta^k(0)-\theta^k(s_k) \right) \notag \\
        &=& \frac{1}{2}\|F(x_k)\|^2-\frac{1}{2}\| F(x_k) + J_k s_k \|^2-\frac{\lambda_k}{2}\|s_k\|^2 \notag \\
        &=& \frac{1}{2}\left(\| F(x_k) \|+\|F(x_k)+J_ks_k\|\right)\left(\|F(x_k)\|-\|F(x_k)+J_ks_k\|\right)-\frac{\mu_k }{2}\|F(x_k)\|^2 \|s_k\|^2 \notag \\
        &\geq& \frac{1}{2} \| F(x_k) \| \left(\|F(x_k)\|-\|F(x_k)+J_ks_k\|\right)-\frac{\mu_k }{2}\|F(x_k)\|^2 \|s_k\|^2 \notag (\text{by \eqref{positibity of diffe in cur and nxt}})\\
        &\geq& \frac{1}{2}c_2 \| R^{-1}_{x_k}(\overline{x_k}) \|\left(\|F(x_k)\|-\|F(x_k)+J_ks_k\|\right) - \frac{\tred{\mumaxz}}{2}L^2\|R^{-1}_{x_k}(\overline{x_k})\|^2\|s_k\|^2 \notag (\text{by Assumption~\ref{Assumption for local zero}(b),(c), Lemma~\ref{local Lipschitz like}} )\notag \\
        & \geq& \frac{1}{2}c_2 \|R^{-1}_{x_k}(\overline{x_k})\|^2\left(c_2-c_4 \|R^{-1}_{x_k}(\overline{x_k})\| \right)-\frac{\tred{\mumaxz}}{2}c_3^2L^2\|R^{-1}_{x_k}(\overline{x_k})\|^4. \label{ denominator eval } (\text{by Assumption~\ref{Assumption for local zero}(b), Lemma~\ref{Lemma 5.1.3}} )
    \end{eqnarray}
Note that
    \begin{alignat}{2}
        &\; \; \; \; \;  \| F(x_k) \| \notag \\
        & \leq  \| J_{k-1} s_{k-1}- (F(x_k) - F(x_{k-1})) \| + \| J_{k-1} s_{k-1} +F(x_{k-1}) \| \notag \\
        & \leq  c_1 \| s_{k-1} \|^2 + c_4 \| R^{-1}_{x_{k-1}} (\overline{x_{k-1}} )\|^2 \notag
        &\quad& (\text{by Assumption~\ref{Assumption for local zero}(a) and \eqref{Lemma 5.1.3-2} }) \notag \\
        & \leq  (c_1 c_3^2 + c_4 ) \| R^{-1}_{x_{k-1}} (\overline{x_{k-1}} )\|^2. \label{next point bound}
        &\quad& (\text{by } \eqref{Lemma 5.1.3-1} \text{ in Lemma~\ref{Lemma 5.1.3}}) 
    \end{alignat}
    The absolute value of the numerator in \eqref{ 1- rho_k} is bounded as
    \begin{eqnarray*}
     && \left| f(R_{x_k}(s_k))-\frac{\theta^k(s_k)}{2}\right| \notag \\
     &=& \frac{1}{2}\left| \left(\|F(R_{x_k}(s_k))\| + \|F(x_k) + J_k s_k \| \right) \left(\|F(R_{x_k}(s_k))\| - \|F(x_k) + J_k s_k \| \right) - \mu_k \|F(x_k)\|^2 \|s_k\|^2 \right| \notag \\
     & \leq & \frac{1}{2}\left( \|F(R_{x_k}(s_k)) \| + \|F(x_k)\|+\|J_k\|\|s_k\|\right)\| J_ks_k-\left(F(R_{x_k}(s_k)-F(x_k))\right) \| + \frac{1}{2}\mu_k \|F(x_k)\|^2 \|s_k\|^2 \notag \\
     & \leq & \frac{1}{2}\left(\|F(R_{x_k}(s_k)) \| + L\|R^{-1}_{x_k}(\overline{x_k})\| + c_3 \| J_k \| \|R^{-1}_{x_k}(\overline{x_k})\|  \right) \| J_ks_k-\left(F(R_{x_k}(s_k)-F(x_k))\right) \| \notag \\
     && + \tblue{\frac{\tred{\mumaxz}}{2}L^2 c_3^2} \|R^{-1}_{x_k}(\overline{x_k})\|^4 \; (\text{by Assumption~\ref{Assumption for local zero} (c), Lemma~\ref{local Lipschitz like}, and \eqref{Lemma 5.1.3-1} })\notag \\
    \end{eqnarray*}
    For the first term, we have
    $\|F(R_{x_k}(s_k)) \| \leq (c_1 c_3^2 +c_4)\|R^{-1}_{x_k}(\overline{x_k})\|^2 $ by noticing $x_{k+1}=R_{x_k}(s_k)$ and using \eqref{next point bound}.
    Furthermore, using   $\|J_k\|\leq K$ by \eqref{upperK} and $\| J_ks_k-\left(F(R_{x_k}(s_k)-F(x_k))\right) \| \leq c_1 c_3^2 \|R^{-1}_{x_k}(\overline{x_k})\|^2$ derived by
 Assumption~\ref{Assumption for local zero}(a) together with \eqref{Lemma 5.1.3-1}, we have
\begin{eqnarray}\label{ numerator eval}
  & & \left| f(R_{x_k}(s_k))-\frac{\theta^k(s_k)}{2}\right| \\
 &  \leq &   \frac{c_1 c_3^2 (L+c_3 K)}{2}\| R^{-1}_{x_k}(\overline{x_k})\|^3 + \frac{c_3^2}{2}\left(c_1(c_1 c_3^2 + c_4) +\tred{\mumaxz}L^2  \right) \| R^{-1}_{x_k}(\overline{x_k})\|^4. \notag
\end{eqnarray}      
  Using (\ref{ denominator eval }) and (\ref{ numerator eval}) for \eqref{ 1- rho_k}, we have
    \begin{eqnarray}
        && | 1- \rho_k | \notag \\
        &\leq& \frac{c_1 c_3^2 (L+c_3 K)\| R^{-1}_{x_k}(\overline{x_k})\| + c_3^2\left(c_1(c_1 c_3^2 + c_4) +\tred{\mumaxz}L^2 c_3^2 \right) \| R^{-1}_{x_k}(\overline{x_k})\|^2}{c_2 \left(c_2-c_4 \|R^{-1}_{x_k}(\overline{x_k})\| \right)-\tred{\mumaxz}c_3^2L^2\|R^{-1}_{x_k}(\overline{x_k})\|^2}\notag \\
        &\leq& \frac{c c_1 c_3^2 (L+c_3 K)\Dist(x_k,X^*) + c^2 c_3^2 \left(c_1(c_1 c_3^2 + c_4) +\tred{\mumaxz}L^2 c_3^2 \right) \Dist(x_k,X^*)^2}{c_2 \left(c_2-c c_4 \Dist(x_k,X^*) \right)-c^2 \tred{\mumaxz}c_3^2L^2\Dist(x_k,X^*)^2}  \notag \\
        &\leq& \frac{c c_1 c_3^2 (L+c_3 K)\dist(x_k,x^*) + c^2 c_3^2 \left(c_1(c_1 c_3^2 + c_4) +\tred{\mumaxz}L^2 c_3^2 \right) \dist(x_k,x^*)^2}{c_2 \left(c_2-c c_4 \dist(x_k,x^*) \right)-c^2 \tred{\mumaxz}c_3^2L^2\dist(x_k,x^*)^2} , \notag \\
        &&  \label{ |1-rhok|}
    \end{eqnarray}
    where the second inequality follows from $ \| R^{-1}_{x_k}(\overline{x_k})\| \leq c \Dist(x_k,X^*) $ by Lemma~\ref{lemma inv retraction} and the last one follows from $\Dist(x_k,X^*) \leq \dist(x_k,x^*) $ by their definitions. 
    Note that the denominator of \eqref{ |1-rhok|} is positive because of the assumption $\dist(x_k,x^*) < c^*$, where $c^*$ is defined by \eqref{c*_def}.
    From \eqref{ |1-rhok|}, it follows that $ | 1- \rho_k | \rightarrow 0 $ as  $\dist(x_k,x^*) \rightarrow 0  $, which implies that 
    there exists some $r^* >0 $ such that if $x_k \in B(x^*,\min{\{\frac{b}{2},r^*\}}) $, then $\rho_k \geq \eta $ holds for the given $\eta \in (0,1) $. 
    
    Therefore, if $\{x_k\} $ satisfies $ x_k \in B(x^*,\min{\{\frac{b}{2},r^*\}})$ and $\dist(x_k,x^*) < c^* $ for all $k \in \{0,1,2,\dots\}$, then all iterations are successful. The proof is complete.
    \hfill$\Box$
    \end{proof}
    
    Hereinafter, we assume 
    \Add{$b \leq \frac{2c_2}{c c_4}$, $ \frac{b}{2} \leq r^*$, and $\frac{b}{2} < c^* $}. This condition is fulfilled by re-taking a sufficiently small $b$ if necessary.
    \begin{lemma}\label{ quadratic reduction of distance}
        \Add{Suppose that Assumption~\ref{Assumption for local zero} holds.}
        If $x_k,x_{k-1} \in B(x^*,\frac{b}{2})$ hold with some $ k \geq 1$, then $ \Dist(x_k,X^*) \leq c_5 \Dist(x_{k-1},X^*)^2$ holds, where $ c_5 := \frac{ \text{\Add{$c^3$}} (c_1 c_3^2+c_4)}{c_2}$.
    \end{lemma}
    \begin{proof}
First of all, note that the $(k-1)$-th iteration is successful by Lemma~\ref{ for large k, kepp to be successful} and $ x_{k-1} \in B(x^*,\frac{b}{2})$.
It follows that 
    \begin{alignat}{2}
     & \; \; \; \; \; \frac{c_2}{c} \; \text{Dist}(x_k,X^*) \notag \\
     &\leq c_2 \; \| R^{-1}_{x_k}(\overline{x_k}) \| \notag 
     &\quad& (\text{by Lemma~\ref{lemma inv retraction} \tred{with $(x,y)=(x_k,\bxk)$}}) \\
     & \leq  \| F(x_k) \| \notag 
     &\quad& (\text{by Assumption~\ref{Assumption for local zero}(b) }) \\
     & \leq  (c_1 c_3^2 + c_4 ) \| R^{-1}_{x_{k-1}} (\overline{x_{k-1}} )\|^2  \notag
     &\quad& (\text{by \eqref{next point bound} \tred{in the proof of Lemma~\ref{ for large k, kepp to be successful}}})\\
     & \leq  c^2 (c_1 c_3^2+c_4)\; \text{Dist}(x_{k-1},X^*)^2. \notag
     &\quad& (\text{by Lemma~\ref{lemma inv retraction} \tred{with $(x,y)=(x_{k-1},\overline{x_{k-1}})$}})
    \end{alignat}
    Therefore, we conclude
    \begin{eqnarray*}
     \text{Dist}(x_k,X^*) \leq c_5 \; \text{Dist}(x_{k-1},X^*)^2.
    \end{eqnarray*}
    \hfill$\Box$
    \end{proof}
    
\tred{In order to prove Lemma~\ref{sufficiently small r}, we will show  
$\dist (x_k, x_{k+1}) \leq c \|s_k\|$ for each $k$. This inequality trivially holds true in the Euclidean case. For the verification of the inequality in the present manifold setting, we need the following lemma concerning $\injR $ that is defined in Definition~\ref{inj of exp}:
}
\begin{lemma} \label{lemm inf inj}
    There exists some $ b^* > 0 $ such that if $ 0 < b \leq b^*$, then
    \begin{eqnarray}
        \inf_{x \in B(x^*,\frac{b}{2})} \injR(x) \geq  \frac{b c c_3}{2} \label{inf inj is positive}
    \end{eqnarray}
    holds.
\end{lemma}

\begin{proof}
    First, we show that there exists some $b^* > 0 $ such that 
    \begin{eqnarray*}
        \inf_{x \in B(x^*,\frac{b^*}{2})} \injR(x) \geq  \frac{b^* c c_3}{2}
    \end{eqnarray*}
    holds.
    To derive a contradiction, suppose that 
    \begin{eqnarray}
        \inf_{x \in B(x^*,\frac{b}{2})} \injR(x) <  \frac{b c c_3}{2} \label{assum 4 (i)}
    \end{eqnarray}
    holds for all $ b > 0$. Let $ \{b_n\}$ be a monotonically decreasing sequence satisfying $b_n \downarrow 0 $. By the assumption, 
    \begin{eqnarray}
        \inf_{x \in B(x^*,\frac{b_n}{2})} \injR(x) <  \frac{b_n c c_3}{2} \label{assum 4 (ii)}
    \end{eqnarray}
    holds for all $n \in \{0,1,2,\dots \}$.
    By \cite[Corollary 10.24]{Boumal}, $\injR \colon \mani \rightarrow (0,\infty]$ is continuous. Combining this with the compactness of $B(x^*, \frac{b_n}{2}) $, we have 
    \begin{eqnarray}
        \inf_{x \in B(x^*,\frac{b_n}{2})} \injR(x) &=& \min_{x \in B(x^*,\frac{b_n}{2})} \injR(x) > 0. \label{assum4 (iii)}
    \end{eqnarray}
    for all $n \in \{0,1,2,\dots \}$.
    Since the left-hand side of \eqref{assum 4 (ii)} is monotonically non-decreasing with respect to $n$, we have
    \begin{eqnarray}
        0 &<& \inf_{x \in B(x^*,\frac{b_0}{2})} \injR(x) \; \; \; \; \; \; \; \; (\text{by } \eqref{assum4 (iii)})\notag  \\
        &\leq& \inf_{x \in B(x^*,\frac{b_n}{2})} \injR(x) \notag \\
        &<& \frac{b_n c c_3}{2} \label{assum4 (5)}
    \end{eqnarray}
    Taking $n \rightarrow \infty $ in \eqref{assum4 (5)} leads to a contradiction as desired. Therefore, there exists some $b^* > 0 $ such that 
    \begin{eqnarray}
        \inf_{x \in B(x^*,\frac{b^*}{2})} \injR(x) \geq \frac{b^* c c_3}{2} \label{assum 4 final}
    \end{eqnarray}
    holds. Moreover, for all $0 < b \leq b^* $, we have 
    \begin{eqnarray*}
        \inf_{x \in B(x^*,\frac{b}{2})} \injR(x) \geq \inf_{x \in B(x^*,\frac{b^*}{2})} \injR(x) \geq \frac{b^* c c_3}{2}  \geq \frac{b c c_3}{2},
    \end{eqnarray*}
    which is the desired assertion.
    \hfill$\Box$
\end{proof}
Hereinafter, we assume that $b$ satisfies $b \leq b^* $ where $b^*$ is the constant in Lemma~\ref{lemm inf inj}.
Using \eqref{inf inj is positive}, we can prove that $\dist (x_k, x_{k+1}) \leq c \|s_k\|$ holds for each $k$.
\begin{lemma} \label{lemma14 prep}
    Under $x_{k} \in B(x^*,\frac{b}{2})$, $\dist (x_k, x_{k+1}) \leq c \|s_k\|$ holds.
\end{lemma} 
\begin{proof}
 Since it holds that 
    \begin{alignat}{2}
        & \; \; \; \; \; \|s_k\|_{x_k} \notag \\
        & \leq c_3 \| R^{-1}_{x_k}(\overline{x_k}) \| 
        &\quad& (\text{by \eqref{Lemma 5.1.3-1}}) \notag \\
        &\leq c c_3 \dist(x_k,\overline{x_k}) 
        &\quad& (\text{by Lemma~\ref{lemma inv retraction}}) \notag \\
        &\leq c c_3 \dist(x_k,x^*) \notag \\
        &\leq \frac{b c c_3}{2} 
        &\quad& (\text{by } x_k \in B\left(x^*,\frac{b}{2}\right))  \notag \\
        &\leq \inf_{x \in B\left(x^*,\frac{b}{2}\right)} \injR (x) 
        &\quad& (\text{by \eqref{inf inj is positive} in Lemma~\ref{lemm inf inj}})\notag \\
        &\leq \injR (x_k) 
        &\quad& (\text{by } x_k \in B\left(x^*,\frac{b}{2}\right)), \notag
    \end{alignat}
we find that $x_{k+1} = R_{x_k}(s_k) \in R(\injR (x_k))$ and thus 
by \eqref{dist and inv retraction}, $\dist (x_k, x_{k+1}) \leq c \|s_k\|$ holds.
The proof is complete.

\hfill$\Box$
\end{proof}

    \begin{lemma} \label{sufficiently small r}
        \Add{Suppose that Assumption~\ref{Assumption for local zero} holds}
        and let $r:=\min{ \left\{ \frac{b}{2+4c^2 c_3}, \frac{1}{2 c_5}\right\} } $. If $x_0 \in B(x^*,r)$ and every iteration is successful, then
        $x_k \in B(x^*,\frac{b}{2})$ holds for all $k \geq 0 $.\\
    \end{lemma}    
    \begin{proof}
When $k=0$, $x_0 \in B(x^*,\frac{b}{2})$ clearly holds since $r \leq \frac{b}{2}$ by definition. 
\tred{In what follows, we show the assertion for $k\ge 1$ by 
  induction.}
    We first show $ x_1 \in B(x^*,\frac{b}{2})$.
\tred{Note that} 
    \begin{alignat}{2}
     & \; \; \; \; \; \text{dist}(x_1,x^*) \notag \\
     &\leq  \text{dist}(x_0,x^*)+\text{dist}(x_0,x_1) \notag \\
     &\leq  \text{dist}(x_0,x^*) + c \| s_0 \| \notag
     &\quad&  (\text{by }  \text{Lemma~\ref{lemma inv retraction} and } s_0 = R_{x_0}^{-1}(x_1).) \\
     & \leq  \text{dist}(x_0,x^*)+c^2 c_3 \text{Dist}(x_0,X^*) \notag 
     &\quad& (\text{by \eqref{Lemma 5.1.3-1} and Lemma~\ref{lemma inv retraction} })\\
     & \leq  (1+c^2 c_3) \text{dist}(x_0,x^*) 
     &\quad& (\text{by } \Dist(x_0,X^*) \leq \dist(x_0,x^*)) \label{align:0605-1}
    \end{alignat}	
          which together with 
 $$ (1+c^2 c_3) \text{dist}(x_0,x^*)  \leq (1+2c^2 c_3) r  \leq\frac{1+2c^2 c_3}{2+4 c^2 c_3} b = \frac{b}{2}$$
          implies
 $x_1 \in B(x^*,\frac{b}{2})$ .     
    
 \tred{Next, \tred{we prove} that $x_{k+1} \in B(x^*,\frac{b}{2})$ for each $k \geq 1$ by supposing that $x_l \in B(x^*,\frac{b}{2})$ \; ($l=0,1,\dots,k)$ holds with some $k \geq \tred{1}$.} 
\tred{By this assumption and Lemma \ref{ quadratic reduction of distance}, we have}
    \begin{eqnarray}
     \text{Dist}(x_l,X^*) \leq c_5 \text{Dist}(x_{l-1},X^*)^2 \leq \dots \leq c_5^{2^l-1} \text{Dist}(x_0,X^*)^{2^l}.\label{ dist bound by pow of 2-1}
    \end{eqnarray}
    Moreover, as $c_5\le \frac{1}{2r}$ by the choice of $r$ and $x_0\in B(x^{\ast},r)$, 
$$
c_5^{2^l-1} \text{Dist}(x_0,X^*)^{2^l}\le 
\left(\frac{1}{2}\right)^{2^l-1}\frac{r^{2^l}}{r^{2^l-1}}
=r\left(\frac{1}{2}\right)^{2^l-1},$$ 
which along with \eqref{ dist bound by pow of 2-1} implies
\begin{equation}
     \text{Dist}(x_l,X^*) \leq r \left(\frac{1}{2}\right)^{2^l-1}\label{ dist bound by pow of 2}
\end{equation}
for each $l=0,1,2,\ldots,k$.
Let us upper-bound $  \text{dist}(x_{k+1},x^*)$ 
    \tred{by the following two-steps:
First, applying the triangle inequality and Lemma~\ref{lemma14 prep} to $\text{dist}(x_{k+1},x^*)$ successively,} we have
    \begin{align}
            \text{dist}(x_{k+1},x^*) &\leq  \dist(x_k,x^*)+\dist(x_k,x_{k+1}) \notag\\
                                          &\leq \text{dist}(x_k,x^*) + c \| s_k \|_{\tblue{x_k}} \notag\\
                                          &\leq  \text{dist}(x_1,x^*)+c \sum_{l=1}^{k} \| s_l \|_{\tblue{x_l}}. \label{al:0605-1}  
 \end{align}
Second, \eqref{al:0605-1} is further bounded as follows:
\begin{align}
\eqref{al:0605-1}
&\le (1+c^2 c_3)r +c \sum_{l=1}^{k} \| s_l \|_{\tblue{x_l}}\notag\\
&\le (1+c^2 c_3)r +c c_3\sum_{l=1}^{k}R_{x_l}^{-1}(\overline{x_l})
\notag\\
&\le (1+c^2 c_3)r +c^2 c_3 \sum_{l=1}^{k} \text{Dist}(x_l,X^*) \notag\\
                       &\leq  (1+c^2 c_3)r + c^2 c_3 r \sum_{l=1}^{k} \left(\frac{1}{2} \right)^{2^l-1},\notag
\end{align}
where 
the first inequality follows from \eqref{align:0605-1}
and the second one does from Lemma\,\ref{Lemma 5.1.3} with $x=x_l$ and 
$x_{l}\in B(x^*,\frac{b}{2})$. 
Moreover, the third and last ones are implied by Lemma\,\ref{lemma inv retraction}
with $x=x_{l}\in B(x^{*},\frac{b}{2})$ and \eqref{ dist bound by pow of 2}, respectively.

Consequently, we have
\begin{equation}
\dist(x_{k+1},x^{\ast})\le 
(1+c^2 c_3)r + c^2 c_3 r \sum_{l=1}^{k} \left(\frac{1}{2} \right)^{2^l-1}.
\label{eq:0605-3}
\end{equation}
 Finally, by $ 2^{l}-1 \geq l$ for $ l \geq 1 $, 
    \begin{eqnarray*}
     \sum_{l=1}^{k} \left(\frac{1}{2} \right)^{2^l-1} 
     &\leq&  \frac{1}{2} + \left(\frac{1}{2}\right)^2 + \dots + \left(\frac{1}{2}\right)^k \notag \\
     &=& 1-\left(\frac{1}{2}\right)^k \leq 1.\notag
    \end{eqnarray*}
    Hence, 
    from this fact and \eqref{eq:0605-3},
    $\text{dist}(x_{k+1},x^*) \leq (1+2c^2c_3)r \leq \frac{b}{2}$ holds and thus $x_{k+1} \in B(x^*,\frac{b}{2})$. 
    The proof is completed.
    \hfill$\Box$
    \end{proof}

    \begin{theorem}\label{local zero theorem}
        Suppose that Assumptions \ref{Assumption for local zero} holds and let $r>0$ be the same as in Lemma \ref{sufficiently small r}. Moreover, assume $x_0 \in B(x^*, r) $.
        Then, $ \{\Dist(x_k,X^*)\}$ converges to $0$ quadratically and furthermore, $\{ x_k \}$ converges to some $ \hat{x} \in B(x^*,\frac{b}{2})$.
    \end{theorem}
    
    \begin{proof} 
    The first assertion can be verified as follows. By the assumptions and Lemma~\ref{sufficiently small r}, $\{x_k\} \subset B(x^*,\frac{b}{2}) $ holds. Thus, we can repeatedly apply Lemma~\ref{ quadratic reduction of distance} and consequently, we conclude that $ \{x_k\}$ quadratically converges to $0$.
    Next, we show the second claim. Since $(\mathcal{M},\langle \cdot,\cdot \rangle)$ is a complete Riemannian manifold and thus it is a complete metric space with respect to $\dist (\cdot,\cdot) $, it suffices to show that $ \{x_k \} $ is a Cauchy sequence.
    
    For arbitrary $m > n$, we have 
    \begin{alignat}{2}	
     & \; \; \; \; \; \text{dist}(x_m,x_n) \notag \\
     & \leq  \sum_{l=n}^{m-1} \text{dist}(x_l,x_{l+1}) \notag \\
     & \leq c \sum_{l=n}^{m-1} \| s_l \|_{\tred{x_l}} \notag 
     &\quad& (\text{by $\{x_k\} \subset B\left(x^*,\frac{b}{2}\right) $ and Lemma~\ref{lemma inv retraction}}) \\
     &\leq   c c_3 \sum_{l=n}^{m-1} \| R^{-1}_{x_l} (\overline{x_l})  \| \notag 
     &\quad&  (\text{by \eqref{Lemma 5.1.3-1}})\\
     & \leq  c^2 c_3 \; \sum_{l=n}^{m-1} \text{Dist}(x_l,X^*) \notag
     &\quad& (\text{by Lemma~\ref{lemma inv retraction}}) \\
     & \leq  c^2 c_3 r \sum_{l=n}^{m-1} \left(\frac{1}{2}\right)^{2^l-1} .\notag
     &\quad& (\text{by \eqref{ dist bound by pow of 2} })
    \end{alignat}
     Using
    $
     \sum_{l=n}^{m-1} (\frac{1}{2})^{2^l-1}  \leq  \sum_{l=n}^{\infty} (\frac{1}{2})^{2l-1} = \frac{1}{3}\left(\frac{1}{2}\right)^{2n-3}
    $, we obtain
    \begin{eqnarray*}
     \text{dist}(x_m,x_n) & \leq & \frac{c^2 c_3 r}{3}\left(\frac{1}{2} \right)^{2n-3}.
    \end{eqnarray*}
    This inequality indicates that $\{x_k\}$ is a Cauchy sequence and consequently, the second claim has been proved. The proof is completed.
    \hfill$\Box$
    \end{proof}
    
    From Theorem~\ref{local zero theorem}, the RLM has a local quadratic convergence property for zero-residual cases. Next, we study the local behavior of our RLM when \eqref{P} is nonzero-residual.

    \subsection{Linear convergence for nonzero-residual cases } \label{local nonzero}
Recall the definitions of $f^{\ast}$ and $X^{\ast}$ in the beginning of subsection\,\ref{sec:notations}.    
In this subsection, we consider the nonzero-residual case, namely, $f^*>0$.

 Besides Assumption\,\ref{Assumption for local zero}, 
\tblue{we suppose that $\text{flag}^{\text{nz}} = \text{true} $ in Algorithm~\ref{alg1} and} further make the following assumptions on $x^* \in X^*$.
    \begin{assume}
         \label{Assumption for local non-zero} \;  
    $\grad f$ is $L_0$-Lipschitz continuous on $B(x^*,b)$.
    \end{assume}
    It is shown in Corollary 10.45 of \cite{Boumal}  that
    this assumption is satisfied when $ f$ is twice continuously differentiable on $ B(x^*,\frac{b}{2})$.
    
We remark that Assumption~\ref{Assumption for local zero} (c) can be removed if the problem is globally nonzero-residual, namely, the global optimal value is larger than 0.
Indeed, as discussed in Remark~\ref{boundMu},
by combining $ \| J(x) \| \leq K$ of \eqref{upperK}  with Assumption~\ref{Assumption for local non-zero} and using an algorithmic parameter \Add{$\beta>1$ and  $ \kappa$ defined by \eqref{defi of kappa} with $L=L_0 $ and $M=K$}, we derive an upper bound of $ \{\mu_k\}$ by
    \begin{eqnarray*}
        \tred{\mumaxnz}:= \frac{\beta \kappa}{2 \underset{x \in \mani}{\min}f(x)}.
    \end{eqnarray*}
    
    
    We begin by introducing a lemma similar to Lemma~\ref{Lemma 5.1.3}.
    \begin{lemma}\label{lem15}
    Suppose $x_k \in B(x^*,\frac{b}{2})$ with some $k$. Under Assumptions \ref{Assumption for local zero} and \ref{Assumption for local non-zero},
    \begin{eqnarray}
     && \| s_k \| \leq \hat{c}_3 \| R^{-1}_{x_k}(\overline{x_k}) \|,  \label{Lemma14 s_k} \\
     && \| J_k s_k+F(x_k)\| - \sqrt{2{f}^*} \leq \hat{c}_4 \| R^{-1}_{x_k}(\overline{x_k}) \|^2 \label{Lemma14 J_k}
    \end{eqnarray}
    hold where
    \Add{$ f^* = f(\overline{x_k}) $,}
    $\hat{c}_3 := \frac{c L_0}{2\mu_{\min}f^*} $, and
    $\hat{c}_4 := \frac{\tred{\mumaxnz}bcL}{2}+ \sqrt{\tred{\mumaxnz}}L + c_1 + \frac{\sqrt{2f^*}\tred{\mumaxnz}}{2} $.
    \end{lemma}
    
    \begin{proof}
We first prove \eqref{Lemma14 J_k}. It follows that $ \theta^k(s_k) \leq \theta^k( R^{-1}_{x_k}(\overline{x_k}))$ from \eqref{ s_k = argmin theta}. 
    Moreover, \eqref{original theta_k_s} immediately yields $ \| J_k s_k + F(x_k) \|^2 \leq \theta^k(s_k)$. Therefore, we have
    \begin{eqnarray}\label{Lemma 5.2.1 (i)}
     \| J_k s_k + F(x_k) \|^2 &\leq& \theta^k(s_k)  \notag \\
     & \leq & \theta^k( R^{-1}_{x_k}(\overline{x_k})) \notag \\
     &=& \|F(x_k)+ J_k R^{-1}_{x_k}(\overline{x_k}) \|^2 + \lambda_k \| R^{-1}_{x_k}(\overline{x_k}) \|^2 .
    \end{eqnarray}
    
    By the same argument as in the proof for Lemma \ref{Lemma 5.1.3}, we have $\overline{x_k} \in B(x^*,b)$.
    Hence, from Assumption~\ref{Assumption for local zero}(a), we obtain
    \begin{eqnarray}
     \| F(x_k) +J_k R^{-1}_{x_k}(\overline{x_k}) \|^2 &\leq& (\| J_k R^{-1}_{x_k}(\overline{x_k}) - (F(\overline{x_k})-F(x_k)) \| +\| F(\overline{x_k}) \| )^2 \notag \\
     & \leq & (c_1 \| R^{-1}_{x_k}(\overline{x_k}) \|^2 + \sqrt{2 f^*})^2. \label{lemma 5.4 (*)}
    \end{eqnarray}
    By combining \eqref{lemma 5.4 (*)} with (\ref{Lemma 5.2.1 (i)}), we get
    \begin{eqnarray}\label{Lemma 5.2.1 (ii)}
     \| J_k s_k + F(x_k) \|^2 \leq (c_1 \| R^{-1}_{x_k}(\overline{x_k}) \|^2 + \sqrt{2 f^*})^2 + \lambda_k \| R^{-1}_{x_k}(\overline{x_k}) \|^2 .
    \end{eqnarray}
    Furthermore, we have $ \sqrt{\frac{\lambda_k}{\mu_k}} = \| F(x_k) \| \leq \| F(x_k)-F(\overline{x_k}) \| + \| F(\overline{x_k}) \| \leq L \; \| R^{-1}_{x_k}(\overline{x_k})\| + \sqrt{2f^*}$ from Lemma\,\tred{\ref{local Lipschitz like}}. Consequently, with $ \mu_k \leq \mumaxnz$, $\lambda_k $ can be bounded from the above as follows:
    \begin{eqnarray}\label{Lemma 5.2.1 (iii)}
     \lambda_k \leq \tred{\mumaxnz}\left( L^2 \; \| R^{-1}_{x_k}(\overline{x_k})\|^2 +2\sqrt{2f^*}L \| R^{-1}_{x_k}(\overline{x_k}) \| + 2 f^* \right).
    \end{eqnarray} 
    
    Applying (\ref{Lemma 5.2.1 (iii)}) to (\ref{Lemma 5.2.1 (ii)}), we find
    \begin{eqnarray}
     && \| J_k s_k + F(x_k) \|^2 \notag \\
     &\leq& (c_1 \| R^{-1}_{x_k}(\overline{x_k}) \|^2 + \sqrt{2 f^*})^2 + \tred{\mumaxnz}L^2 \; \| R^{-1}_{x_k}(\overline{x_k})\|^4 +2\sqrt{2f^*} \tred{\mumaxnz}L \| R^{-1}_{x_k}(\overline{x_k}) \|^3 + 2 f^* \tred{\mumaxnz}\| R^{-1}_{x_k}(\overline{x_k}) \|^2 \notag \\
     &=& (c_1^2 + \tred{\mumaxnz}L^2) \| R^{-1}_{x_k}(\overline{x_k}) \|^4 + 2 \sqrt{2f^*} \tred{\mumaxnz}L \| R^{-1}_{x_k}(\overline{x_k}) \|^3 + (2\sqrt{2f^*}c_1+2f^* \tred{\mumaxnz})\| R^{-1}_{x_k}(\overline{x_k}) \|^2 + 2f^* . \notag \\
     && \label{0811 (i)}
    \end{eqnarray}
Moreover, since $ \| R^{-1}_{x_k}(\overline{x_k}) \| \leq c \; \text{dist}(x_k,\overline{x_k}) \leq c \; \text{dist}(x_k,x^*) \leq \frac{bc}{2} $ holds by Lemma~\ref{lemma inv retraction} and the assumption $x_k \in B(x^*,\frac{b}{2}) $, we have
    \begin{eqnarray}
     && 2 \sqrt{2f^*} \tred{\mumaxnz}L \| R^{-1}_{x_k}(\overline{x_k}) \|^3 + (2\sqrt{2f^*}c_1+2f^* \overset{\rm{nz}}{\mu_{\max}})\| R^{-1}_{x_k}(\overline{x_k}) \|^2 \notag \\
     &\leq & (\sqrt{2f^*} \tred{\mumaxnz}bcL + 2\sqrt{2f^*}c_1+2f^*\tred{\mumaxnz}) \| R^{-1}_{x_k}(\overline{x_k}) \|^2. \label{0811 (ii)}
    \end{eqnarray}
   \tred{\eqref{0811 (i)} and \eqref{0811 (ii)} yield}
    \begin{eqnarray*}
     && \| J_ks_k + F(x_k) \|^2 \\
     &\leq& (c_1^2 + \tred{\mumaxnz}L^2) \| R^{-1}_{x_k}(\overline{x_k}) \|^4 + (\sqrt{2f^*} \tred{\mumaxnz}bcL + 2\sqrt{2f^*}c_1+2f^*\tred{\mumaxnz}) \| R^{-1}_{x_k}(\overline{x_k}) \|^2 + 2f^* \\
     &\leq& \left( \left(\frac{\tred{\mumaxnz}bcL}{2}+ \sqrt{\mumaxnz}L + c_1 + \frac{\sqrt{2f^*}\tred{\mumaxnz}}{2} \right)\| R^{-1}_{x_k}(\overline{x_k}) \|^2 + \sqrt{2f^*} \right)^2\\
     &=&\tred{\left(\hat{c}_4\| R^{-1}_{x_k}(\overline{x_k})\|^2+ \sqrt{2f^*} \right)^2}.
    \end{eqnarray*}
    Therefore, we have 
    $
     \| J_k s_k + F(x_k) \| - \sqrt{2f^*} \leq \hat{c}_4 \; \| R^{-1}_{x_k}(\overline{x_k})\|^2.
    $ 
    
    In turn, we show $\| s_k \| \leq \hat{c}_3 \| R^{-1}_{x_k}(\overline{x_k}) \|$.
    \begin{alignat}{2}
     & \; \; \; \; \; \|s_k \| \notag \\
     &\leq \frac{1}{\lambda_k} \| \text{grad}f(x_k) \| \notag 
     &\quad&  (\text{by } \eqref{s_k grad rela})\\
     & \leq  \frac{L_0}{2\mu_{\min} f^*}\text{Dist}(x_k,X^*) \notag 
     &\quad& (\text{by Assumption~\ref{Assumption for local non-zero} and } \lambda_k = \mu_k \|F(x_k)\|^2 \geq 2\mu_{\min}f^* ) \\ 
     & \leq  \frac{c L_0}{2\mu_{\min}f^*} \| R^{-1}_{x_k}(\overline{x_k}) \| = \hat{c}_3 \| R^{-1}_{x_k}(\overline{x_k}) \|. \notag 
     &\quad& (\text{by Lemma~\ref{lemma inv retraction} })
    \end{alignat}
    Hence, the proof is completed.
    \hfill$\Box$
    \end{proof}
    
    When we fix $\mu_k $ as $ \tred{\mumaxnz}$ for all $ k \in \{0,1,2,\dots\}$, every iteration is successful. We need this property for establishing the local convergence as 
    \Add{shown} below. Hereinafter we set $ \mu_k = \tred{\mumaxnz}$. Note that this is consistent with Algorithm \ref{alg1} since $ \mu_k = \tred{\mumaxnz}$ holds for all $k \in \{0,1,2,\dots\} $ by setting $ \mu_{\min} = \tred{\mumaxnz}$.
    \begin{lemma}\label{Lemma 5.2.3}
        Suppose that  Assumptions\,\ref{Assumption for local zero} and \ref{Assumption for local non-zero} hold.
        Moreover, if $ f^* < \frac{c_2^4}{8 c^{8} c_1^2}$, and $x_k,x_{k-1} \in B(x^*,\frac{b}{2})$ hold, then it follows that  
        $$ \Dist(x_k,X^*) \leq c_5 \Dist(x_{k-1},X^*),$$ 
        where 
        $$
         c_5 := \sqrt{\frac{c^4(c_1{\hat{c}_3}^2+\hat{c}_4)\left( \frac{b^2 c^2 (c_1 {\hat{c}_3}^2 + \hat{c}_4)}{4} + 2\sqrt{2f^*}\right)}{c_2^2 - 2 \sqrt{2f^*}c^4 c_1}}.$$ 
    \end{lemma}
    
    \begin{proof}
    First, we have 
    \begin{alignat}{2}\label{Lemma 5.2.3 (i)}
     & \; \; \; \; \; c_2 \text{Dist}(x_k,X^*) \notag \\
     & \leq  c c_2 \| R^{-1}_{x_k}(\overline{x_k}) \|
     &\quad& (\text{by Lemma~\ref{lemma inv retraction}})\notag \\
     &\leq c \| F(x_k)-F(\overline{x_k}) \|
     &\quad& (\text{by Assumption~\ref{Assumption for local zero}(b)}).
    \end{alignat}
    Moreover, 
    \tred{the following inequality holds:}
    \begin{alignat}{2}
     & \; \; \; \; \; \|F(x_k) \| \notag \\
     & \leq  \| J_{k-1} s_{k-1}- (F(x_k)-F(x_{k-1}))) \| + \| J_{k-1} s_{k-1} + F(x_{k-1}) \| \notag \\
     & \leq  c_1 \| s_{k-1} \|^2 + \hat{c}_4 \| R^{-1}_{x_{k-1}} (\overline{x_{k-1}} )\|^2 + \sqrt{2f^*} \notag
     &\quad& (\text{by Assumption~\ref{Assumption for local zero}(a) and \eqref{Lemma14 J_k}}) \notag \\
     & \leq (c_1 \hat{c}_3^2 + \hat{c}_4 ) \| R^{-1}_{x_{k-1}} (\overline{x_{k-1}} )\|^2 + \sqrt{2f^*} \notag 
     &\quad& (\text{by \eqref{Lemma14 s_k}})\notag \\
     & \leq  c^2(c_1\hat{c}_3^2+\hat{c}_4)\; \text{Dist}(x_{k-1},X^*)^2 + \sqrt{2f^*} \notag
     &\quad& (\text{by Lemma~\ref{lemma inv retraction}})\notag \\
     &= A+\sqrt{2f^*}, \label{ local nonzero F(x_k) bound}
    \end{alignat}
    where $A := c^2(c_1 \hat{c}_3^2+\hat{c}_4)\; \text{Dist}(x_{k-1},X^*)^2 $. \\
    
    Next, we evaluate $\| F(x_k)-F(\overline{x_k}) \| $ by using \eqref{ local nonzero F(x_k) bound}. Then, we find 
    \begin{eqnarray}\label{Lemma 5.2.3 (ii)}
     \| F(x_k)-F(\overline{x_k}) \|^2 
     &=& \| F(x_k) \|^2 -2F(\overline{x_k})^T F(x_k) + 2f^* \notag \\
     &\leq& A^2 + 2\sqrt{2f^*}A + 4f^*-2F(\overline{x_k})^T F(x_k).
    \end{eqnarray}
    \Add{Now}, we analyze $ 4f^*-2F(\overline{x_k})^T F(x_k)$ \Add{in \eqref{Lemma 5.2.3 (ii)}}. 
    Let
    \begin{eqnarray}
        r_{x_k} := - J(\overline{x_k})R^{-1}_{\overline{x_k}}(x_k) + (F(x_k)-F(\overline{x_k}))  \notag
    \end{eqnarray}
    so as to satisfy 
    \begin{eqnarray}
        F(x_k) = F(\overline{x_k})+ J(\overline{x_k})R^{-1}_{\overline{x_k}}(x_k) + r_{x_k}. \label{ r_{x_k} }
    \end{eqnarray}
    Then, from Assumption~\ref{Assumption for local zero}(a) and Lemma~\ref{lemma inv retraction}, we have
    \begin{eqnarray}
     \| r_{x_k} \| \leq c_1 \| R^{-1}_{\overline{x_k}}(x_k) \|^2 \leq c^2 c_1 \text{Dist}(x_k,X^*)^2, \label{r_{x_k} inequ}
    \end{eqnarray}
    which yields
    \begin{alignat}{2}\label{Lemma 5.2.3. (iii)}
     & \; \; \; \; \; 4f^*-2F(\overline{x_k})^T F(x_k) \notag \\
     &= 2F(\overline{x_k})^T F(\overline{x_k})-2F(\overline{x_k})^T F(x_k) \notag \\
     &= -2F(\overline{x_k})^T (F(x_k)-F(\overline{x_k})) \notag \\
     &= -2 F(\overline{x_k})^T J(\overline{x_k}) R^{-1}_{\overline{x_k}}(x_k)-2F(\overline{x_k})^T r_{x_k} \notag
     &\quad& (\text{by \eqref{ r_{x_k} } }) \notag \\
     &= -2 \langle {J(\overline{x_k})}^* F(\overline{x_k}), R^{-1}_{\overline{x_k}}(x_k) \rangle_{\overline{x_k}} - 2F(\overline{x_k})^T r_{x_k} 
     &\quad& (\text{\tred{by \eqref{def of adj}}}) \notag \\
     &= -2 \langle \text{grad}f(\overline{x_k}), R^{-1}_{\overline{x_k}}(x_k) \rangle_{\overline{x_k}}-2F(\overline{x_k})^T r_{x_k} \notag \\
     &= -2 F(\overline{x_k})^T r_{x_k}  \notag
     &\quad& (\text{by } \text{ grad}f(\overline{x_k}) = 0_{\overline{x_k}})\notag \\
     &\leq 2 \| F(\overline{x_k}) \| \| r_{x_k} \| \notag \\
     &\leq 2 \sqrt{2f^*} \; c^2 c_1 \; \text{Dist}(x_k,X^*)^2.
     &\quad&  (\text{by \eqref{r_{x_k} inequ}})
    \end{alignat}
    
    From (\ref{Lemma 5.2.3. (iii)}) and (\ref{Lemma 5.2.3 (ii)}), we obtain
    \begin{eqnarray}\label{Lemma5.2.3 (iv)}
     \| F(x_k)-F(\overline{x_k}) \|^2 \leq A^2 + 2\sqrt{2f^*}A + 2 \sqrt{2f^*} \; c^2 c_1 \; \text{Dist}(x_k,X^*)^2 .
    \end{eqnarray}
    
    \Add{Using (\ref{Lemma5.2.3 (iv)}) in the right-hand side of (\ref{Lemma 5.2.3 (i)})}
    , we find
    \begin{eqnarray*}
     && c_2^2 \; \text{Dist}(x_k,X^*)^2 \leq c^2 \left( A^2 + 2\sqrt{2f^*}A + 2 \sqrt{2f^*} \; c^2 c_1 \; \text{Dist}(x_k,X^*)^2 \right),
    \end{eqnarray*}
    implying 
    \begin{eqnarray}
        (c_2^2 - 2 \sqrt{2f^*}c^4 c_1) \text{Dist}(x_k,X^*)^2 \leq c^2 (A+2\sqrt{2f^*})A. \label{Lemma16 last part}
    \end{eqnarray}
    Using the definition of $ A$ together with $\Dist (x_{k-1},X^*) \leq \frac{b}{2}$, \eqref{Lemma16 last part} leads to
    \begin{eqnarray*}
     &&(c_2^2 - 2 \sqrt{2f^*}c^4 c_1) \text{Dist}(x_k,X^*)^2 \\
     &\leq& c^4(c_1{\hat{c}_3}^2+\hat{c}_4)\left( \frac{b^2 c^2 (c_1 {\hat{c}_3}^2 + \hat{c}_4)}{4} + 2\sqrt{2f^*}\right) \text{Dist}(x_{k-1},X^*)^2.
    \end{eqnarray*}
    Since $ f^* < \frac{c_2^4}{8 c^{8} c_1^2}$
    holds by the assumption, we have $ c_2^2 - 2 \sqrt{2f^*}c^4 c_1 >0$. Therefore, 
    we obtain
    \begin{eqnarray*}
     \text{Dist}(x_k,X^*) \leq c_5 \text{Dist}(x_{k-1},X^*).
    \end{eqnarray*}
    The proof is completed.
    \hfill$\Box$
    \end{proof}
    
    \Add{Note that $ c,c_1$ and $c_2$ are constants independent of $f^*$ and hence the condition $f^* < \frac{c_2^4}{8 c^{8} c_1^2} $ in Lemma~\ref{Lemma 5.2.3} makes sense. Considering characteristics of these constants, $f^*$ tends to be required to be relatively small so as to satisfy this condition. 
    Moreover, even though $c_5<1$ is required to establish the linear convergence of Algorithm~\ref{alg1}, we are not sure about how reasonable this condition is since $c_5$ depends on constants appearing in Assumption~\ref{Assumption for local zero}. 
    Hereinafter, We assume, however, that $c_5<1$ is satisfied in addition to $f^* < \frac{c_2^4}{8 c^{8} c_1^2} $ for our analysis.}

    \begin{lemma}\label{Lemma 5.2.4}
        Suppose that Assumptions\, \ref{Assumption for local zero}, \ref{Assumption for local non-zero} , and $c_5 < 1$ hold and define $r:=\frac{b}{2}(1+\frac{c^2 \hat{c}_3}{1-c_5})^{-1}$. If $x_0 \in B(x^*,r)$, then
        $x_k \in B(x^*,\frac{b}{2})$ holds for all $k \in \{0,1,2,\dots\} $. \\
    \end{lemma}
    
    \begin{proof} 
        When $k=0$, $x_{0} \in B(x^*,\frac{b}{2})$ holds clearly since $r \leq \frac{b}{2}$ by the definition.  For $k\ge 1$, we prove $x_{k} \in B(x^*,\frac{b}{2})$ by induction. First, we consider $k=1$.
    Noting
    \begin{alignat}{2}
     & \; \; \; \; \; \text{dist}(x_1,x^*) \notag \\
     &\leq  \text{dist}(x_0,x^*)+\text{dist}(x_0,x_1) \notag \\
     &\leq  r+ c\| s_0 \| \notag 
     &\quad& (\text{by } x_0 \in B(x^*,r) \text{ , Lemma~\ref{lemma14 prep}, and Lemma~\ref{lemma inv retraction}}) \notag \\
     & \leq  r+c^2 \hat{c}_3 \; \text{Dist}(x_0,X^*) \notag
     &\quad& (\text{by \eqref{Lemma14 s_k} and Lemma~\ref{lemma inv retraction}}) \notag \\
     & \leq  r+c^2 \hat{c}_3 \; \text{dist}(x_0,x^*) \leq (1+c^2 \hat{c}_3)r. \label{ x_1 contained}
    \end{alignat}
    and $ (1+c^2 \hat{c}_3) r = \frac{b}{2}(1+c^2 \hat{c}_3)(1+\frac{c^2 \hat{c}_3}{1-c_5})^{-1} \leq \frac{b}{2}$ hold, we find that $x_1 \in B(x^*,\frac{b}{2})$. 
    
    In what follows, supposing $x_l \in B(x^*,\frac{b}{2})$ \; ($l=0,1,\dots,k)$ holds for some $k \geq 1$, we show $x_{k+1} \in B(x^*,\frac{b}{2})$.
  By these assumptions together with Lemma~\ref{Lemma 5.2.3}, we have
    \begin{eqnarray}
     \text{Dist}(x_l,X^*) \leq c_5 \text{Dist}(x_{l-1},X^*) \leq \dots \leq c_5^{l} \; \text{Dist}(x_0,X^*) \leq r c_5^l \label{Dist pow bounded}
    \end{eqnarray}
    for all $1 \leq l \leq k$, which yields
    \begin{alignat}{2}
      & \; \; \; \; \; \text{dist}(x_{k+1},x^*) \notag \\
      & \leq  \dist (x_k,x^*) + \dist (x_k,x_{k+1}) \notag \\
      & \leq  \text{dist}(x_k,x^*) + c \| s_k \|_{x_k} \notag
      &\quad& (\text{by Lemma~\ref{lemma inv retraction}})\\
      & \leq  \text{dist}(x_{k-1},x^*)+c \| s_{k-1} \|_{x_{k-1}} + c \| s_k \|_{x_k}\notag \\
      &\quad\quad\quad\quad\quad\vdots\notag\\
      & \leq  \text{dist}(x_1,x^*)+c \sum_{l=1}^{k} \| s_l \|_{x^l} \notag \\
      & \leq  (1+c^2 \hat{c}_3)r +c^2 \hat{c}_3 \sum_{l=1}^{k} \text{Dist}(x_l,X^*) \notag
      &\quad& (\text{by \eqref{ x_1 contained}, \eqref{Lemma14 s_k}, and Lemma~\ref{lemma inv retraction}})\\
      & \leq  (1+c^2 \hat{c}_3 )r + c^2 \hat{c}_3 r \sum_{l=1}^{k} c_5^l, \notag 
      &\quad&  (\text{by \eqref{Dist pow bounded}})
    \end{alignat} 
    which together with $\sum_{l=1}^{k} c_5^l \leq \sum_{l=1}^{\infty} c_5^l = \frac{c_5}{1-c_5}$ implies
    \begin{eqnarray*}
     \text{dist}(x_{k+1},x^*) \leq \left(1+\frac{c^2 \hat{c}_3}{1-c_5} \right)r = \frac{b}{2},
    \end{eqnarray*}
    thus we conclude $x_{k+1} \in B(x^*,\frac{b}{2})$. Hence, we obtain the desired assertion.      
     
%
%

    \hfill$\Box$
    \end{proof}

    \begin{theorem}
        Suppose that $ f^* < \frac{c_2^4}{8 c^{8} c_1^2}$, $c_5 < 1$ and  
        Assumptions \ref{Assumption for local zero}, \ref{Assumption for local non-zero} hold. Let $r>0$ be the same as in Lemma \ref{Lemma 5.2.4}. Moreover, assume $x_0 \in B(x^*,r)$. 
        Then, $ \{\text{Dist}(x_k,X^*)\}$ converges to $0$ linearly and furthermore, $ \{ x_k \}$ converges to some point $ \hat{x} \in B(x^*,\frac{b}{2})$. 
    \end{theorem}
    
    \begin{proof} 
        The former claim follows from the assumptions, Lemma \ref{Lemma 5.2.3}, and Lemma \ref{Lemma 5.2.4}. We next show the latter one. Using the assumption that $(\mathcal{M},\langle \cdot , \cdot \rangle)$ is a complete Riemannian manifold, it suffices to show that $\{x_k \}$ is a Cauchy sequence with respect to the Riemannian distance $\dist(\cdot,\cdot)$. For arbitrary $m > n$, we have
    \begin{alignat}{2}
     & \; \; \; \; \; \text{dist}(x_m,x_n) \notag \\
     & \leq  \sum_{l=n}^{m-1} \text{dist}(x_l,x_{l+1}) \notag \\
     & \leq  c \sum_{l=n}^{m-1} \| s_l \|_{x_l} \notag 
     &\quad& (\text{by Lemma~\ref{lemma inv retraction}}) \\
     & \leq c^2 \hat{c}_3 \; \sum_{l=n}^{m-1} \text{Dist}(x_l,X^*) \notag 
     &\quad& (\text{by \eqref{Lemma14 s_k} and Lemma~\ref{lemma inv retraction}})\\
     & \leq  c^2 \hat{c}_3 r \sum_{l=n}^{m-1} c_5^l, \notag 
     &\quad& (\text{by \eqref{Dist pow bounded}})
    \end{alignat}
    which together with
    $
     \sum_{l=n}^{m-1} c_5^l	 \leq  \sum_{l=n}^{\infty} c_5^l = \frac{c_5^n}{1-c_5}
    $ implies
    \begin{eqnarray*}
     \text{dist}(x_m,x_n) & \leq & \frac{c^2 \hat{c}_3 r}{1-c_5} \; c_5^n .
    \end{eqnarray*}
    Therefore, noting $0<c_5<1$, we ensure that $ \{x_k \} $ is a Cauchy sequence. The proof is complete.
    \hfill$\Box$
    \end{proof}

\section{Numerical experiments} \label{chap:experiment}
    
We apply the proposed RLM to two kinds of problems:  CANDECOMP/PARAFAC (CP) decomposition of tensors and low-rank matrix completion. 
All experiments were conducted on a machine with an Intel Core i5 CPU and 8.0 GB RAM. Regarding implementations, all methods were implemented in Matlab. 

\subsection{CP decomposition of tensors}

Here we apply the RLM method to the CP decomposition of tensors.

\subsubsection{Brief introduction to tensor rank approximation problem (TAP)}

Let $S_1$ denote the set of rank one tensors of format $n_1 \times \dots \times n_d$, $S^r := \overbrace{S_1 \times \dots \times S_1}^{r} $, and
let $\| \cdot \|_{\rm F} $ denote the Frobenius norm. For a given tensor $\mathcal{A} \in \mathbb{R}^{n_1 \times \dots \times n_d}$ and $r>0$, rank $r$ CP decomposition of $\mathcal{A} $ is formulated as the following optimization problem with the map 
$ \Phi(p): (p_1,\dots,p_r)  \longmapsto  \sum_{i=1}^r p_i$, where $p_i \in S_1$, $i=1,\ldots,r$.

\begin{eqnarray*}
 \begin{aligned}
 (\text{TAP}) \; \; & \underset{ p \in S^r} {\text{min}} && f(p) = \frac{1}{2}\| {\Phi(p)- \mathcal{A}\|_{ \rm F}}^2.
 \end{aligned}
\end{eqnarray*}

For solving (\text{TAP}), the article\,\cite{TAP} proposed the trust-region-based Riemannian Gauss-Newton method ``RGN-HR'' with a manipulation named ``hot-restart'' specialized for solving TAP.
We utilize the same retraction and geometry as their work and compare performances of our proposal, RGN-HR, and RGN (i.e., RGN without hot-restart), where we used the Matlab code\footnote{\url{https://arxiv.org/src/1709.00033v2/anc}} provided by \cite{TAP} for RGN-HR. 
Since RGN can frequently encounter ill-conditioned linear equations, as a remedy, we employ the Moore-Penrose pseudo-inverse matrix in solving them.  

\subsubsection{Experimental setting of TAP}

We sampled a tensor $\mathcal{A} \in \mathbb{R}^{13 \times 11 \times 9} $ or $\mathcal{A} \in \mathbb{R}^{50 \times 50 \times 50} $ from 
``Model 2'' in \cite{TAP} and generated an input tensor $\mathcal{B} $ according to $\mathcal{B} = \frac{\mathcal{A}}{\| \mathcal{A} \|_{\rm F}} + 10^{-\text{\Add{$p$}}}\frac{\mathcal{E}}{\| \mathcal{E} \|_{\rm F}}$, where $ \mathcal{E}$ is a tensor with the same size as $\mathcal{A}$, whose each element of the tensor is independently and identically distributed random variable from $\mathcal{N}(0,1) $. The parameter $\text{\Add{$p$}}$ controls the degree of perturbation. \\

In this experiment, $\text{\Add{$p$}}$ is fixed as $\text{\Add{$p$}}=5$ and the decomposition rank is $r=5$.
As the hyperparameters  $ \eta,\mu_{\min},$ and $\beta$ in Algorithm~\ref{alg1}, we set $ \eta = 0.2, \mu_{\min} = 0.1, \beta = 5.0$ in RLM.
As the stopping rule, we make each algorithm terminate when any one of the following conditions 
\Add{is} satisfied:

\begin{description}
\item[(c1)] The iteration number exceeds $ \text{MAX ITER} = 1000$.
\item[(c2)] $f(x_k) \leq 10^{-10}$ holds.
\item[(c3)] $ \| \grad f(x_k) \| \leq 10^{-6}$ holds.
\end{description}

\subsubsection{Comparison by averaged performances}
To compare averaged performances of our RLM, RGN-HR and RGN,\footnote{
Unlike the low-rank matrix completion problems solved later on,   
we do not select solvers from \textit{Manopt} as competitors. 
This is because it does not provide tools such as a retraction for dealing with the manifold $S^r$ as of the time of writing this paper.}
we generated 10 tensors $\mathcal{B}_i $ ($ 1 \leq i \leq 10$) in the above way and we set 50 randomized starting points for each tensor. 
We show the results of our experiments in Table \ref{TAP}, where their each row represents the following:
\begin{itemize}
  \item
    success: the number of runs terminated due to fulfilling the stopping rules (c2) or (c3) among 500 runs. The left and right numbers in the parentheses
    show the number of iterations terminated due to the stopping rule (c2) and 
    the stopping rule (c3), respectively. 
    \item
    fail: the left and right numbers show the number of iterations terminated due to the stopping rule (c1) and due to some numerical error, respectively.
  \item
    $t_{\text{success}}$: the averaged computational time among the successful runs.
\end{itemize}

\begin{table}[tb]
  \centering
  \caption{Comparison of RLM, RGN, and RGN-HR}
      \label{TAP}
    \begin{tabular}{c||c|c|c}
      \multicolumn{4}{l}{$\mathcal{B}_i \in \mathbb{R}^{13 \times 11 \times 9},~ i=1,\cdots, 10$} \\ \hline
        & RLM &  RGN & RGN-HR \\
        \hline
        success & $493 ~(375:118)$ & $475 ~(203:272)$ & $500 ~(500:0)$ \\
        fail & $7:0$ & $25:0$ & $0:0$  \\
        $t_{\text{success}} $ (sec.) & $1.232$ & $1.674$  & $6.519 \times 10^{-1}$ \\
        \hline
        \multicolumn{4}{l}{ } \\
         \multicolumn{4}{l}{$\mathcal{B}_i \in \mathbb{R}^{50 \times 50 \times 50},~ i=1,\cdots, 10$} \\ \hline
        & RLM & RGN & RGN-HR \\
        \hline
        success & $492 ~(335:157)$  & $481 ~(242:239)$ &  $495 ~(485:10)$ \\
        fail & $8:0$   & $14:5$ & $0:5$ \\
       $t_{\text{success}} $ (sec.) & $1.088  \times 10 $  & $3.616 \times 10$ &  $3.150$  \\
        \hline
    \end{tabular}
\end{table}

\Add{As Table~\ref{TAP} shows, RLM outperforms RGN in all items, though it is defeated by RGN-HR, which is specialized for solving TAP without any theoretical guarantees. 
  RGN-HR and RGN contain five instances in which they could not reach MAX ITER without the stopping rules satisfied.
  In all such cases, the progress in the computation stalled at the calculation of the retraction which uses the sequentially-truncated higher order singular value decomposition (ST-HOSVD) \cite{hosvd,st-hosvd}. Thus, we infer they were provoked due to numerically unstable calculations of the ST-HOSVD.
 While RLM employs the same retraction as RGN-HR and RGN, it did not cause such an instance as long as we experimented. 
 These observations may support the stability of RLM in comparison with RGN-HR and RGN.
 }

\begin{figure}[tb]
    \centering
    \includegraphics[width=0.8\textwidth]{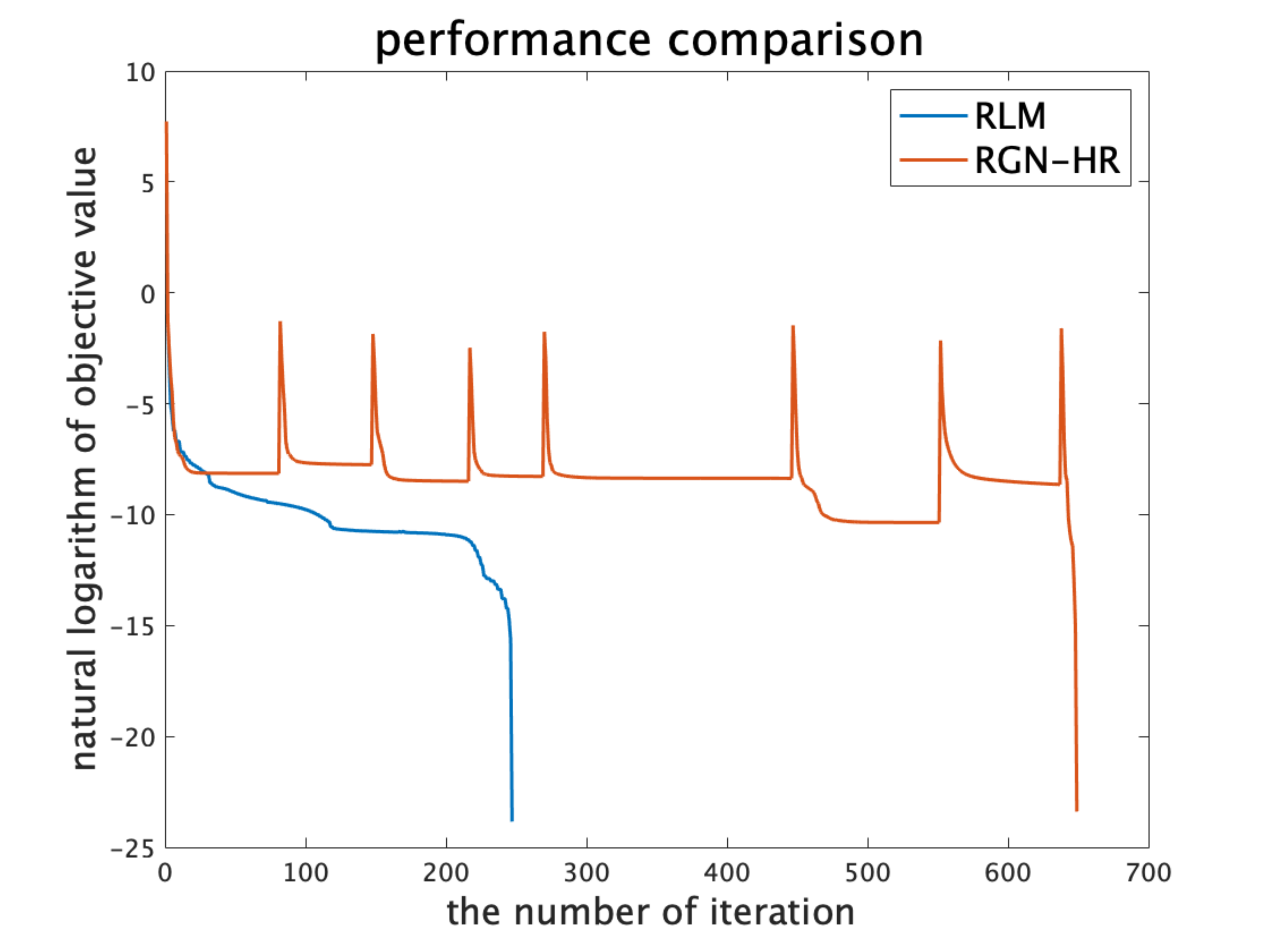}
    \caption{CP decomposition for $ \mathcal{B} \in \mathbb{R}^{13 \times 11 \times 9}$}
    \label{Fig:tensor decomp}
\end{figure}

Figure~\ref{Fig:tensor decomp} shows an example of the change of the objective value of RLM and RGN-HR as the iteration proceeds for $ \mathcal{B} \in \mathbb{R}^{13 \times 11 \times 9}$.
As this figure indicates, $ \{f(x_k)\}$ is monotonically non-increasing as proved in the proof of Theorem~\ref{ global lim inf}. 
Moreover, in this example, the objective value of RLM starts to drastically decrease when it gets relatively small. In most cases, we observed this tendency for RLM.

\subsection{Low-rank matrix completion}

Next, we apply the RLM method to low-rank matrix completion problems. 
    
\subsubsection{Brief introduction to low-rank matrix completion}
This problem is to recover a low-rank matrix from a matrix, say $A \in \mathbb{R}^{m \times n}$, whose elements are known only partially in advance. 
Specifically, letting $\mathbb{R}^{m \times n}_k$ be the set of $m \times n $ matrices with rank $k$, the problem is formulated as follows:
\begin{align}
    \begin{split}
    \underset{X \in \mathbb{R}^{m \times n}_k }{\text{minimize}} \quad & f(x) := \frac{1}{2}\| P_{\Omega}(X) - A \|_{\mathrm{F}}^2,  \label{opt problem for low-rank completion} \\
    \end{split}
\end{align}
where $ \Omega \subset \{1,\dots,m\} \times \{1,\dots,n\}$ denotes the set of indices for which elements of $A $ are known in advance and $ P_{\Omega} \colon \mathbb{R}^{m \times n} \rightarrow \mathbb{R}^{m \times n}$ is given by 
\begin{eqnarray*}
    P_{\Omega}(X) = 
    \begin{cases}
        X_{i,j} & ( (i,j) \in \Omega) \\
        0 & ( \text{otherwise })
    \end{cases}.
\end{eqnarray*}
It is known that $\mathbb{R}^{m \times n}_k $ has a structure as a $ k (m+n-k)$-dimensional smooth manifold embedded into $\mathbb{R}^{m \times n} $ (e.g. \cite{Boumal}). Thus, \eqref{opt problem for low-rank completion} can be regarded as a least-square Riemannian optimization problem.

\subsubsection{Experimental setting} \label{Experimental setting}
We compare the RLM with other four Riemannian methods: Riemannian trust-region (RTR) method with Gauss-Newton approximation for its Hessian approximation, Riemannian gradient descent (RSD) method, Riemannian conjugate gradient(RCG) method provided by {\it Manopt} \cite{boumal2014manopt}, which is a Matlab optimization toolbox on Riemannian manifolds, and adaptive quadratically regularized Newton (ARNT) method proposed by \cite{quadraticallyreguralizedNewton}. We refer to them respectively as ``manoptRTR'', ``manoptRSD'', ``manoptRCG'', and ``ARNT''. 

Given natural numbers $m,n,k$, the oversampling factor $r_s$ (i.e., the ratio of observed elements in $A$) for a low-rank matrix completion is defined as 
\begin{eqnarray*}
    r_s := \frac{|\Omega|}{k(m+n-k)}.
\end{eqnarray*}
Once $r_s$ is given, we set $\Omega$ by repeatedly sampling $(i,j) \in \{1,\dots,m\} \times \{1,\dots,n\} $ such that $(i,j) \notin \Omega $ and adding it to $\Omega$ until $ |\Omega| $ gets equal to $ k(m+n-k)r_s$.

Next, we generate an input matrix $A \in \mathbb{R}^{m \times n}$ in the following manner:
First, we sample $A_L \in \mathbb{R}^{m \times k}$ and $A_R \in \mathbb{R}^{n \times k}$ such that their each element independently and identically follows $\mathcal{N}(0,1) $ and secondly, 
define $A$ as $P_{\Omega}(A_L {A_R}^T)$.

As the parameters  $ \eta,\mu_{\min},$ and $\beta$ in Algorithm~\ref{alg1}, we set $ \eta = 0.2, \mu_{\min} = 0.1, \beta = 5.0$.
As the stopping rule, we make each algorithm terminate when any one of the following conditions is satisfied:    
    \begin{description}
        \item[(c1)] The CPU time exceeds $300$ seconds.
        \item[(c2)] $ \|\grad f(x_k)\| \leq 10^{-8}$ holds.
    \end{description}
In the same manner as in \cite{quadraticallyreguralizedNewton}, we generate an initial point as follows:
With matrices $A_L$ and $A_R$ sampled in the same way as in producing the matrix $A$ above, we first compute $A_L A_R^T $, from which we run manoptRSD to gain a refined point such that the norm of Riemannian gradient gets less than or equal to $10^{-3}$.
The last point is used as an initial solution. This procedure is executed for the sake of observing the performance of our algorithm when the residual of \eqref{P} is sufficiently small.
    
    \subsubsection{Comparison by averaged performances} 
    We compare averaged performances of those methods among 10 starting points generated in the way described in Section \ref{Experimental setting} in the following two types of setting of $m,n,k$, and $r_s$:
    \begin{description}
    \item[(I)] $r_s = 0.9 + 0.01i \; (0\leq i < 10) $, $m=n=30, k=3$
    \item[(II)]$r_s=1.2$, $m = n = 200 + 200i \; (0 \leq i < 5) $, $k=\frac{m}{10}$
    \end{description}
    The setting (I) aims to examine the performances against different $r_s$s with fixed $(m,n,k)$, while (II) for different $(m,n,k)$ with fixed $r_s$.
We evaluate the quality of performances in terms of the following criteria:
\begin{quote}
  \begin{center}
\begin{tabular}{ll}
  {\bm $\text{success}$}: &the number of runs terminated due to fulfilling\\
                          &the stopping rules (c2)\\
  {\bm $\text{iter}_{\text{success}}$}: &the averaged number of iterations among \\
                                     &the successful runs\\
  {\bm $t_{\text{success}}$}:& the averaged computational time among\\
                          &the successful runs.
\end{tabular}
\end{center}
\end{quote}
Figure~\ref{Fig:comparison Low-rank matrix completion - Small but Difficult} shows, from top to bottom, the changes of success, $\text{iter}_{\text{success}}$ (log-scale) and $t_{\text{success}}$ (log-scale) versus the ratio $r_s$ in the setting (I).
    Note that manoptRSD has zero success in all $r_s$ as the top of Figure~\ref{Fig:comparison Low-rank matrix completion - Small but Difficult} shows and thus $\text{iter}_{\text{success}}$ and $t_{\text{success}}$ of manoptRSD cannot be computed. Due to this issue, manoptRSD does not appear in the second and third plots of Figure~\ref{Fig:comparison Low-rank matrix completion - Small but Difficult}.

    Figure\,\ref{Fig:comparison Low-rank matrix completion - Small but Difficult} indicates that RLM has superiority over manoptRTR, manoptRSD, and manoptRCG in terms of both computational time and iteration number. Moreover, RLM is more robust against the change of $r_s$ compared with those methods. 
    In the comparison of RLM and ARNT, RLM still shows superiority over ARNT.
    
    \begin{figure}[tb]
        \centering
        \includegraphics[width=1.2\textwidth]{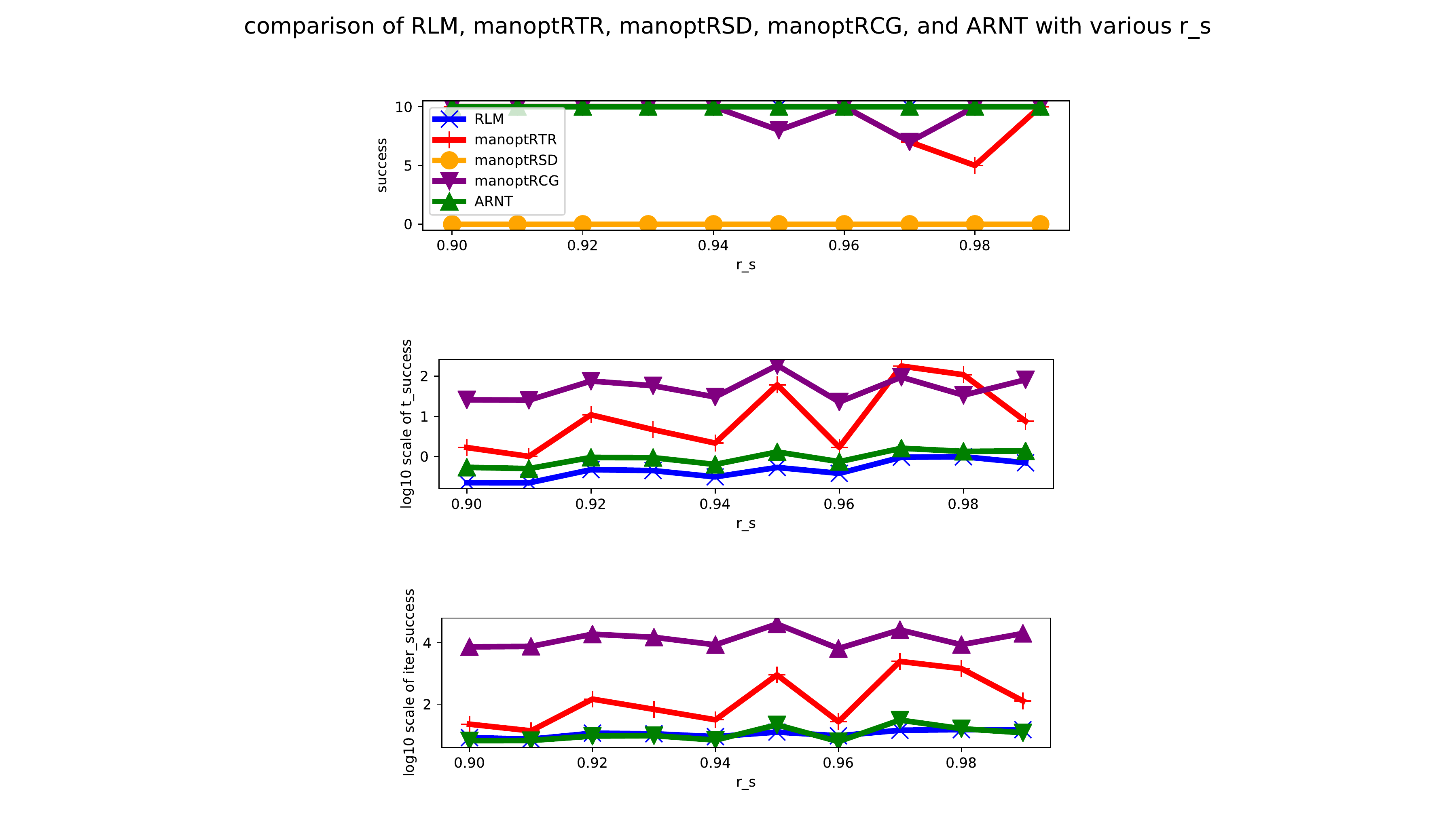}
         \caption{Comparison of RLM to existing methods in setting (I).
         In the middle and bottom figures, manoptRSD and manoptRCG do not appear because of failures for all the instances.
          }
        \label{Fig:comparison Low-rank matrix completion - Small but Difficult}
    \end{figure}

    Figure~\ref{Fig:comparison Low-rank matrix completion - Large} shows the averaged performances in the setting (II). Since RLM and ARNT show very similar performances, some plots of their results overlap in the figure. 
    According to Figure~\ref{Fig:comparison Low-rank matrix completion - Large}, from the perspective of both computational time and iteration number, RLM is as efficient as manoptRTR and ARNT, and is superior to the other methods. 
    \begin{figure}[tb]
        \centering
        \includegraphics[width=1.3\textwidth]{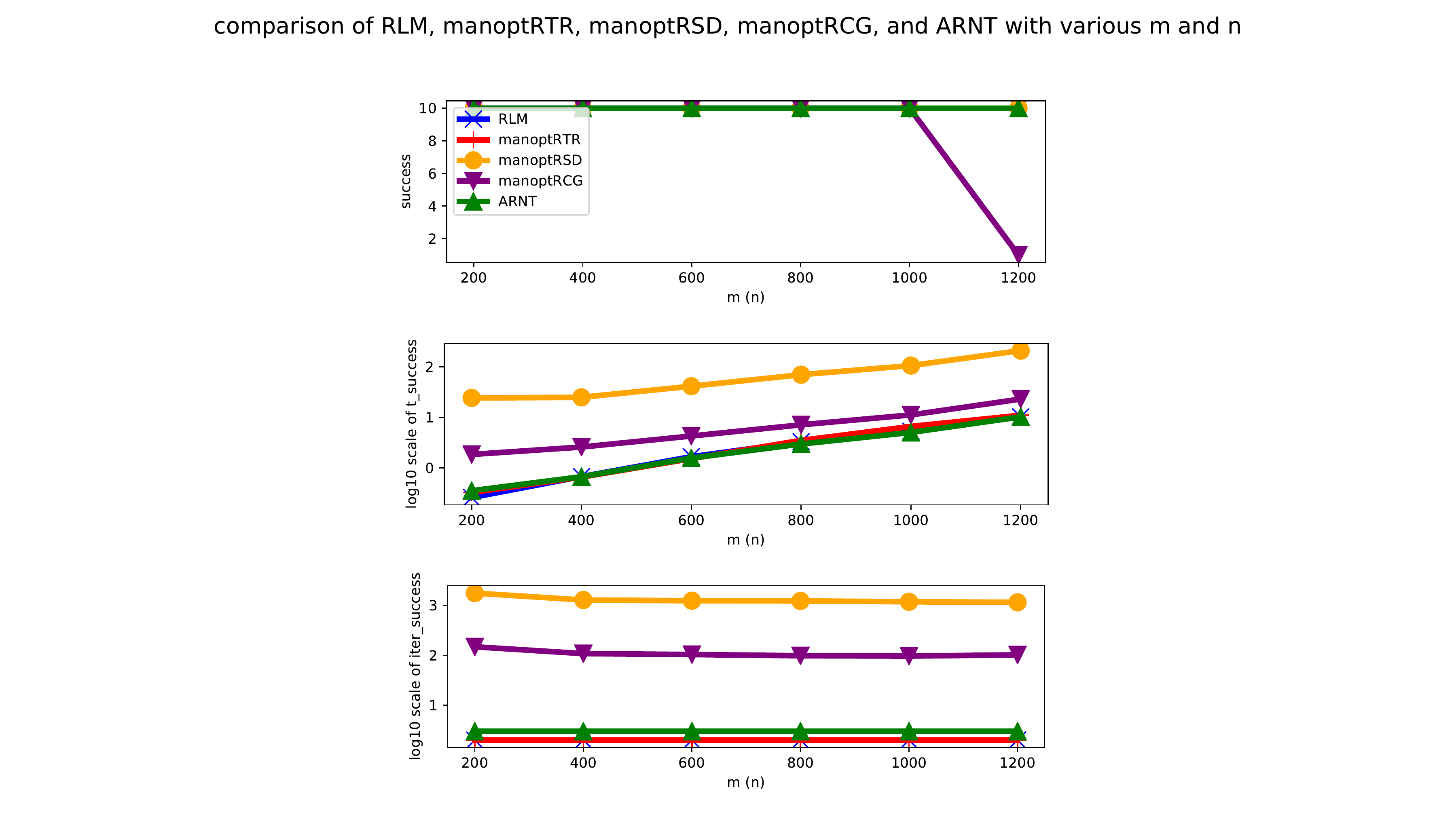}
        \caption{Comparison of RLM to existing methods in setting (II) (the same legend with Figure~\ref{Fig:comparison Low-rank matrix completion - Small but Difficult}) 
        the results of RLM are covered with those for manoptRTR and ARNT in the top and middle figures and those for manoptRTR in the bottom one.
        } 
        \label{Fig:comparison Low-rank matrix completion - Large}
    \end{figure}

    \begin{figure}[tb]
        \centering
        \includegraphics[width=1.0\textwidth]{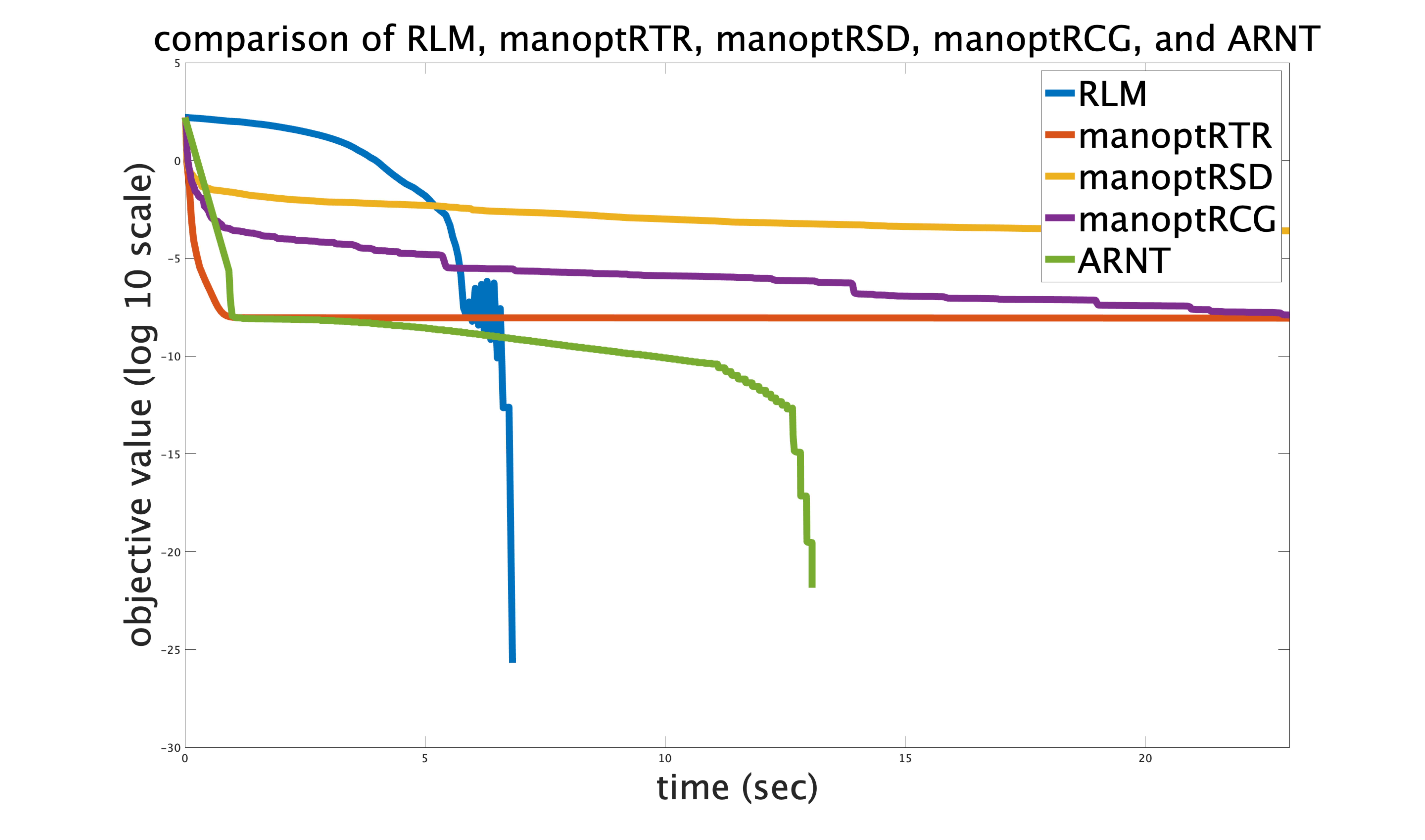}
        \caption{The objective value versus computation time for $m = n = 30 $, $k=3$, $r_s = 0.97 $.}
        \label{Fig:Low-rank matrix completion}
    \end{figure}  
    Figure~\ref{Fig:Low-rank matrix completion} illustrates how each method decreases the objective value against CPU time in an instance of problem with $ m = n = 30, r_s = 0.97$. While manoptRTR, manoptRSD, and manoptRCG get stuck at some point, RLM accomplishes a considerable reduction.
\section{Conclusion} \label{chap:conclusion}
    
    We proposed a Riemannian Levenberg-Marquardt (RLM) method for the nonlinear least-squares problem on the Riemannian manifold of the form \eqref{P}. We proved the global and local convergence properties of the algorithm and 
    conducted two types of numerical experiments: the CP decomposition of tensors and the low-rank matrix completion.
    In both of them, we found the RLM efficiently converges when the residual of \eqref{P} is sufficiently small.

    Possible directions of future work would be to 
    extend the theoretical guarantees to \eqref{P} where $\mani $ is a manifold with boundary.
    Furthermore, we are interested in the establishment of a theory pertaining to the desirable affine transformation for RLM and its relation with the Riemannian metric.

\section*{Compliance with Ethical Standards}
This work was partially supported by the Japan Society for the Promotion of Science KAKENHI Grant Number 19H04069, 20K19748, 20H04145, and 23H03351. There is no conflict of interest in writing the paper.

\bibliographystyle{spmpsci}      

\begin{thebibliography}{10}
\providecommand{\url}[1]{{#1}}
\providecommand{\urlprefix}{URL }
\expandafter\ifx\csname urlstyle\endcsname\relax
  \providecommand{\doi}[1]{DOI~\discretionary{}{}{}#1}\else
  \providecommand{\doi}{DOI~\discretionary{}{}{}\begingroup
  \urlstyle{rm}\Url}\fi

\bibitem{Riemanniantrustregion1}
Absil, P.A., Baker, C.G., Gallivan, K.A.: Trust-region methods on {R}iemannian
  manifolds.
\newblock Foundations of Computational Mathematics \textbf{7}, 303--330 (2007)

\bibitem{RSDRN}
Absil, P.A., Mahony, R., Sepulchre, R.: Optimization algorithms on matrix
  manifolds.
\newblock Princeton University Press (2008)

\bibitem{baker2008implicit}
Baker, C.G., Absil, P.A., Gallivan, K.A.: An implicit trust-region method on
  {R}iemannian manifolds.
\newblock IMA journal of numerical analysis \textbf{28}(4), 665--689 (2008)

\bibitem{Behling}
Behling, R., Gon{\c{c}}alves, D.S., Santos, S.A.: Local convergence analysis of
  the {L}evenberg--{M}arquardt framework for nonzero-residue nonlinear
  least-squares problems under an error bound condition.
\newblock Journal of Optimization Theory and Applications \textbf{183}(3),
  1099--1122 (2019)

\bibitem{Bergou}
Bergou, E.H., Diouane, Y., Kungurtsev, V.: Convergence and complexity analysis
  of a {L}evenberg--{M}arquardt algorithm for inverse problems.
\newblock Journal of Optimization Theory and Applications \textbf{185},
  927--944 (2020)

\bibitem{Bertsekas}
Bertsekas, D.P.: Nonlinear programming, 2nd eds.
\newblock Athena Scientific, Belmont (2003)

\bibitem{boumal2015riemannian}
Boumal, N.: Riemannian trust regions with finite-difference {H}essian
  approximations are globally convergent.
\newblock In: Geometric Science of Information: Second International
  Conference, GSI 2015, Palaiseau, France, October 28-30, 2015, Proceedings 2,
  pp. 467--475. Springer (2015)

\bibitem{Boumal}
Boumal, N.: An introduction to optimization on smooth manifolds.
\newblock Cambridge University Press (2023)

\bibitem{RTRconvergencerate}
Boumal, N., Absil, P.A., Cartis, C.: Global rates of convergence for nonconvex
  optimization on manifolds.
\newblock IMA Journal of Numerical Analysis \textbf{39}(1), 1--33 (2019)

\bibitem{boumal2014manopt}
Boumal, N., Mishra, B., Absil, P.A., Sepulchre, R.: Manopt, a Matlab toolbox
  for optimization on manifolds.
\newblock The Journal of Machine Learning Research \textbf{15}(1), 1455--1459
  (2014)

\bibitem{RGN}
Breiding, P., Vannieuwenhoven, N.: Convergence analysis of {R}iemannian
  {G}auss--{N}ewton methods and its connection with the geometric condition
  number.
\newblock Applied Mathematics Letters \textbf{78}, 42--50 (2018)

\bibitem{TAP}
Breiding, P., Vannieuwenhoven, N.: A {Riemannian} trust region method for the
  canonical tensor rank approximation problem.
\newblock SIAM Journal on Optimization \textbf{28}(3), 2435--2465 (2018)

\bibitem{hosvd}
De~Lathauwer, L., De~Moor, B., Vandewalle, J.: A multilinear singular value
  decomposition.
\newblock SIAM journal on Matrix Analysis and Applications \textbf{21}(4),
  1253--1278 (2000)

\bibitem{mukFk22}
Fan, J.: The modified {Levenberg-Marquardt} method for nonlinear equations with
  cubic convergence.
\newblock Mathematics of Computation \textbf{81}(277), 447--466 (2012)

\bibitem{mukFk23}
Fan, J.: Accelerating the modified {L}evenberg-{M}arquardt method for nonlinear
  equations.
\newblock Mathematics of Computation \textbf{83}(287), 1173--1187 (2014)

\bibitem{mukFk24}
Fan, J., Pan, J.: Convergence properties of a self-adaptive
  {L}evenberg-{M}arquardt algorithm under local error bound condition.
\newblock Computational Optimization and Applications \textbf{34}(1), 47--62
  (2006)

\bibitem{mukFk21}
Fan, J.y.: A modified {Levenberg-Marquardt} algorithm for singular system of
  nonlinear equations.
\newblock Journal of Computational Mathematics pp. 625--636 (2003)

\bibitem{geodesicregressionproposed1}
Fletcher, T.: Geodesic regression on {R}iemannian manifolds.
\newblock In: Proceedings of the Third International Workshop on Mathematical
  Foundations of Computational Anatomy-Geometrical and Statistical Methods for
  Modelling Biological Shape Variability, pp. 75--86 (2011)

\bibitem{originoffrechetmean}
Fr\'{e}chet, M.: Les \'el\'ements al\'eatoires de nature quelconque dans un
  espace distanci\'e.
\newblock Annales de l'institut Henri Poincar\'e \textbf{10}(4), 215--310
  (1948)

\bibitem{riemanndistandlocalcoordinate}
Gallivan, K.A., Qi, C., Absil, P.A.: A {Riemannian Dennis-Mor{\'e}} condition.
\newblock High-Performance Scientific Computing: Algorithms and Applications
  pp. 281--293 (2012)

\bibitem{quadraticallyreguralizedNewton}
Hu, J., Milzarek, A., Wen, Z., Yuan, Y.: Adaptive quadratically regularized
  {N}ewton method for {R}iemannian optimization.
\newblock SIAM Journal on Matrix Analysis and Applications \textbf{39}(3),
  1181--1207 (2018)

\bibitem{Riemanniantrustregion2}
Huang, W.: Optimization algorithms on {R}iemannian manifolds with applications.
\newblock Ph.D. thesis, The Florida State University (2013)

\bibitem{Riemanniantrustregion3}
Huang, W., Absil, P.A., Gallivan, K.A.: A Riemannian symmetric rank-one
  trust-region method.
\newblock Mathematical Programming \textbf{150}(2), 179--216 (2015)

\bibitem{Riemannianquasi-Newton1}
Huang, W., Absil, P.A., Gallivan, K.A.: A {R}iemannian {BFGS} method without
  differentiated retraction for nonconvex optimization problems.
\newblock SIAM Journal on Optimization \textbf{28}(1), 470--495 (2018)

\bibitem{Riemannianquasi-Newton2}
Huang, W., Gallivan, K.A., Absil, P.A.: A Broyden class of quasi-{N}ewton
  methods for {R}iemannian optimization.
\newblock SIAM Journal on Optimization \textbf{25}(3), 1660--1685 (2015)

\bibitem{locallinear}
Ipsen, I.C., Kelley, C., Pope, S.: Rank-deficient nonlinear least squares
  problems and subset selection.
\newblock SIAM Journal on Numerical Analysis \textbf{49}(3), 1244--1266 (2011)

\bibitem{lai2022superlinear}
Lai, Z., Yoshise, A.: Riemannian Interior Point Methods for Constrained
  Optimization on Manifolds.
\newblock arXiv 2203.09762  (2022)

\bibitem{Levenberg}
Levenberg, K.: A method for the solution of certain non-linear problems in
  least squares.
\newblock Quarterly of applied mathematics \textbf{2}(2), 164--168 (1944)

\bibitem{li2021nonmonotone}
Li, X., Wang, X., Krishan~Lal, M.: A Nonmonotone Trust Region Method for
  Unconstrained Optimization Problems on Riemannian Manifolds.
\newblock Journal of Optimization Theory and Applications \textbf{188},
  547--570 (2021)

\bibitem{cliu}
Liu, C., Boumal, N.: Simple algorithms for optimization on {R}iemannian
  manifolds with constraints.
\newblock Applied Mathematics \& Optimization \textbf{82}, 949--981 (2020)

\bibitem{Marquardt}
Marquardt, D.W.: An algorithm for least-squares estimation of nonlinear
  parameters.
\newblock Journal of the society for Industrial and Applied Mathematics
  \textbf{11}(2), 431--441 (1963)

\bibitem{marumo2023majorization}
Marumo, N., Okuno, T., Takeda, A.: Majorization-minimization-based
  {Levenberg--Marquardt} method for constrained nonlinear least squares.
\newblock Computational Optimization and Applications \textbf{84}, 1--42 (2023)

\bibitem{geodesicregressionproposed2}
Niethammer, M., Huang, Y., Vialard, F.X.: Geodesic regression for image
  time-series.
\newblock In: Medical Image Computing and Computer-Assisted
  Intervention--MICCAI 2011: 14th International Conference, Toronto, Canada,
  September 18-22, 2011, Proceedings, Part II 14, pp. 655--662. Springer (2011)

\bibitem{obara}
Obara, M., Okuno, T., Takeda, A.: Sequential quadratic optimization for
  nonlinear optimization problems on {R}iemannian manifolds.
\newblock SIAM Journal on Optimization \textbf{32}(2), 822--853 (2022)

\bibitem{globalRLM}
Osborne, M.: Nonlinear least squares - the {L}evenberg algorithm revisited.
\newblock The ANZIAM Journal \textbf{19}(3), 343--357 (1976)

\bibitem{RLM0}
Peeters, R.L.M.: On a {R}iemannian version of the {L}evenberg-{M}arquardt
  algorithm (1993).
\newblock \urlprefix\url{https://EconPapers.repec.org/RePEc:vua:wpaper:1993-11}

\bibitem{RiemannianCG1}
Ring, W., Wirth, B.: Optimization methods on {R}iemannian manifolds and their
  application to shape space.
\newblock SIAM Journal on Optimization \textbf{22}(2), 596--627 (2012)

\bibitem{RiemannianCG2}
Sato, H.: A {D}ai--{Y}uan-type {R}iemannian conjugate gradient method with the
  weak {W}olfe conditions.
\newblock Computational optimization and Applications \textbf{64}, 101--118
  (2016)

\bibitem{RiemannianCG3}
Sato, H., Iwai, T.: A new, globally convergent {R}iemannian conjugate gradient
  method.
\newblock Optimization \textbf{64}(4), 1011--1031 (2015)

\bibitem{schiela2021sqp}
Schiela, A., Ortiz, J.: An {SQP} method for equality constrained optimization
  on {H}ilbert manifolds.
\newblock SIAM Journal on Optimization \textbf{31}(3), 2255--2284 (2021)

\bibitem{inverseproblem}
Tarantola, A.: Inverse problem theory and methods for model parameter
  estimation.
\newblock SIAM (2005)

\bibitem{globalRLMtrust}
Ueda, K., Yamashita, N., et~al.: On a Global Complexity Bound of the
  {L}evenberg-{M}arquardt Method.
\newblock J. Optim. Theory Appl. \textbf{147}(3), 443--453 (2010)

\bibitem{low-rankmatrixcompletion}
Vandereycken, B.: Low-rank matrix completion by {R}iemannian optimization.
\newblock SIAM Journal on Optimization \textbf{23}(2), 1214--1236 (2013)

\bibitem{st-hosvd}
Vannieuwenhoven, N., Vandebril, R., Meerbergen, K.: A new truncation strategy
  for the higher-order singular value decomposition.
\newblock SIAM Journal on Scientific Computing \textbf{34}(2), A1027--A1052
  (2012)

\bibitem{yamakawa}
Yamakawa, Y., Sato, H.: Sequential optimality conditions for nonlinear
  optimization on {R}iemannian manifolds and a globally convergent augmented
  {L}agrangian method.
\newblock Computational Optimization and Applications \textbf{81}(2), 397--421
  (2022)

\bibitem{Yamashita}
Yamashita, N., Fukushima, M.: On the rate of convergence of the
  {L}evenberg-{M}arquardt method.
\newblock In: G.~Alefeld, X.~Chen (eds.) Topics in Numerical Analysis: With
  Special Emphasis on Nonlinear Problems, pp. 239--249. Springer (2001)

\bibitem{regression}
Yao, Q., Tong, H.: Asymmetric least squares regression estimation: a
  nonparametric approach.
\newblock Journal of nonparametric statistics \textbf{6}(2-3), 273--292 (1996)

\bibitem{nonlineareq}
Yuan, Y.X.: Trust region algorithms for nonlinear equations.
\newblock Hong Kong Baptist University, Department of Mathematics (1994)

\end{thebibliography}

\end{document}